\documentclass[11pt,reqno,twoside]{amsart}
\usepackage{amssymb,amsmath,amsthm,soul,color,paralist}
\usepackage{t1enc}
\usepackage[cp1250]{inputenc}
\usepackage{a4,indentfirst,latexsym}
\usepackage{graphics}
\usepackage{mathrsfs}
\usepackage{cite,enumitem,graphicx}
\usepackage[colorlinks=true,urlcolor=blue,
citecolor=red,linkcolor=blue,linktocpage,pdfpagelabels,
bookmarksnumbered,bookmarksopen]{hyperref}
\usepackage[english]{babel}
\usepackage[left=2.00cm, right=2.00cm, top=1.65cm, bottom=1.65cm]{geometry}
\usepackage[metapost]{mfpic}
\usepackage[hyperpageref]{backref}
\usepackage[colorinlistoftodos]{todonotes}
\usepackage[normalem]{ulem}

\makeatletter
\providecommand\@dotsep{5}
\def\listtodoname{List of Todos}
\def\listoftodos{\@starttoc{tdo}\listtodoname}
\makeatother

\numberwithin{equation}{section}

\newcommand{\h}{H^{s}_{\e}}
\newcommand{\R}{\mathbb{R}}
\newcommand{\2}{2^{*}_{s}}

\newcommand{\C}{\mathbb{C}}

\newcommand{\N}{\mathcal{N}}

\DeclareMathOperator{\dive}{div}
\DeclareMathOperator{\supp}{supp}
\DeclareMathOperator{\e}{\varepsilon}

\newtheorem{prop}{Proposition}[section]
\newtheorem{lem}{Lemma}[section]
\newtheorem{thm}{Theorem}[section]

\newtheorem{cor}{Corollary}[section]

\keywords{Fractional magnetic operators; Kirchhoff equation; variational methods.}
\subjclass[2010]{35A15, 35R11, 35S05, 58E05.}

\date{}

\begin{document}
\title[fractional Kirchhoff equations with magnetic fields]{Multiple concentrating solutions for a fractional Kirchhoff equation with magnetic fields}

\author[V. Ambrosio]{Vincenzo Ambrosio}
\address{Vincenzo Ambrosio\hfill\break\indent 
Dipartimento di Scienze Matematiche, Informatiche e Fisiche \hfill\break\indent
Universit\`a di Udine \hfill\break\indent
via delle Scienze 206 \hfill\break\indent
33100 Udine, Italy \hfill\break\indent}
\email{vincenzo.ambrosio2@unina.it}

\begin{abstract}
This paper is concerned with the multiplicity and concentration behavior of nontrivial solutions for the following fractional Kirchhoff equation in presence of a magnetic field:
\begin{equation*}
\left(a\varepsilon^{2s}+b\varepsilon^{4s-3} [u]_{A/\varepsilon}^{2}\right)(-\Delta)_{A/\varepsilon}^{s}u+V(x)u=f(|u|^{2})u \quad \mbox{ in } \mathbb{R}^{3},
\end{equation*}
where $\varepsilon>0$ is a small parameter, $a, b>0$ are constants, $s\in (\frac{3}{4}, 1)$, $(-\Delta)^{s}_{A}$ is the fractional magnetic Laplacian, $A:\mathbb{R}^{3}\rightarrow \mathbb{R}^{3}$ is a smooth magnetic potential, $V:\mathbb{R}^{3}\rightarrow \mathbb{R}$ is a positive continuous potential having a local minimum and $f:\mathbb{R}\rightarrow \mathbb{R}$ is a $C^{1}$ subcritical nonlinearity. 
Applying penalization techniques and Ljusternik-Schnirelman theory, we relate the number of nontrivial solutions with the topology of the set where the potential $V$ attains its minimum. 
\end{abstract}

\maketitle

\section{introduction}

\noindent
In this paper, we focus our attention on the following fractional Kirchhoff equation
\begin{equation}\label{P}
\left(a\e^{2s}+b\e^{4s-3}[u]_{A/\varepsilon}^{2}\right)(-\Delta)_{A/\varepsilon}^{s}u+V(x)u=f(|u|^{2})u \quad \mbox{ in } \mathbb{R}^{3},
\end{equation}
where $\e>0$ is a small parameter, $a$ and $b$ are positive constants, $s\in (\frac{3}{4}, 1)$, 
$$
[u]_{A/\varepsilon}^{2}:=\iint_{\R^{6}} \frac{|u(x)-e^{\imath (x-y)\cdot \frac{A}{\e}(\frac{x+y}{2})} u(y)|^{2}}{|x-y|^{3+2s}} \, dxdy ,
$$
the function $V: \R^{3}\rightarrow \R$ is a continuous potential verifying the following assumptions introduced by del Pino and Felmer \cite{DF}:
\begin{compactenum}[$(V_1)$]
\item $\inf_{x\in \R^{3}} V(x)=V_{0}>0$;
\item  there exists a bounded domain $\Lambda\subset \R^{3}$ such that
\begin{equation}
V_{0}<\min_{\partial \Lambda} V \quad \mbox{ and } \quad M=\{x\in \Lambda: V(x)=V_{0}\}\neq \emptyset,
\end{equation}
\end{compactenum} 
and $f:\R\rightarrow \R$ is a $C^{1}$-function satisfying the following conditions:
\begin{compactenum}[$(f_1)$]
\item $f(t)=0$ for $t\leq 0$ and $\displaystyle{\lim_{t\rightarrow 0} \frac{f(t)}{t}=0}$;
\item there exists $q\in (4, 2^{*}_{s})$, with $2^{*}_{s}= \frac{6}{3-2s}$, such that 
$$
\lim_{t\rightarrow \infty} \frac{f(t)}{t^{\frac{q-2}{2}}}=0;
$$
\item there exists $\theta\in (4, \2)$ such that $0<\frac{\theta}{2} F(t)\leq t f(t)$ for any $t>0$, where $F(t)=\int_{0}^{t} f(\tau)d\tau$;
\item there exist $\sigma\in (4, \2)$ and $C_{\sigma}>0$ such that $f'(t)t-f(t)\geq C_{\sigma}t^{\frac{\sigma-2}{2}}$ for all $t>0$.
\end{compactenum} 
We assume that $A: \R^{3}\rightarrow \R^{3}$ is a H\"older continuous magnetic potential of exponent $\alpha\in(0,1]$,
and $(-\Delta)^{s}_{A}$ is the fractional magnetic Laplacian which,  up to a normalization constant, is defined for any $u\in C^{\infty}_{c}(\R^{3}, \C)$ as
\begin{equation*}
%\label{operator}
(-\Delta)^{s}_{A}u(x)
:=2 \lim_{r\rightarrow 0} \int_{B_{r}^{c}(x)} \frac{u(x)-e^{\imath (x-y)\cdot A(\frac{x+y}{2})} u(y)}{|x-y|^{3+2s}} \,dy.
\end{equation*}
This operator has been recently introduced in \cite{DS} and relies essentially on the L\'evy-Khintchine formula for the generator of a general L\'evy process. For completeness, we emphasize that in the literature there are three different fractional magnetic operators and that they coincide when $s=1/2$ and $A$ is assumed to be linear;  see \cite{I10} for more details.

In absence of the magnetic field, i.e. $A\equiv 0$, the operator $(-\Delta)^{s}_{A}$ is consistent 
%reduces to the well-known 
with the following definition of fractional Laplacian operator $(-\Delta)^{s}$ for smooth functions $u$
\begin{equation*}
%\label{operator}
(-\Delta)^{s}u(x)
:=2 \lim_{r\rightarrow 0} \int_{B_{r}^{c}(x)} \frac{u(x)-u(y)}{|x-y|^{3+2s}} \,dy.
\end{equation*}
%which may be viewed as the infinitesimal generator of a L\'evy stable diffusion processes \cite{App}, and 
This operator arises in a quite natural way in many different physical situations in which one has to consider long range anomalous diffusions and transport in highly heterogeneous medium; see \cite{DPV}.
%appears in the study of several phenomena in the applied sciences, such as, among others, phase transitions, conservation laws, crystal dislocation, water waves, plasma physics and so on; see \cite{DPV} for more details.
When $s\rightarrow 1$, the authors in \cite{PSV, SV} showed that $(-\Delta)^{s}_{A}$ can be considered as  the fractional counterpart of the magnetic Laplacian
$$
-\Delta_{A} u:=\left(\frac{1}{\imath}\nabla-A\right)^{2}u= -\Delta u -\frac{2}{\imath} A(x) \cdot \nabla u + |A(x)|^{2} u -\frac{1}{\imath} u \dive(A(x)),
$$ 
which plays a fundamental role in quantum mechanics in the description of the dynamics of the particle in a non-relativistic setting; see \cite{RS}. Motivated by this fact, many authors  \cite{AFF, AS, CS, EL, K} dealt with the existence of nontrivial solutions of the following Schr\"odinger equation with magnetic field 
\begin{equation}\label{MSE}
-\e^{2}\Delta_{A} u+V(x)u=f(x, |u|^{2})u \quad \mbox{ in } \R^{N}.
\end{equation}
Equation \eqref{MSE} appears when we seek standing wave solutions $\psi(x, t)=u(x)e^{-\imath \frac{E}{\e}t}$, with $E\in \R$, for the following time-dependent nonlinear Schr\"odinger equation with magnetic field:
$$
\imath \e \frac{\partial \psi}{\partial t}=\left(\frac{\e}{\imath}\nabla-A(x)\right)^{2} \psi+U(x)\psi-f(|\psi|^{2})\psi \quad \mbox{ in } (x, t)\in \R^{N}\times \R,
$$
where $U(x)=V(x)+E$. An important class of solutions of \eqref{MSE} are the so called semi-classical states which concentrate and develop a spike shape around one, or more, particular points in $\R^{N}$, while vanishing elsewhere as $\e\rightarrow 0$. This interest is due to the fact that the transition from quantum mechanics to classical mechanics can be formally performed by sending $\e\rightarrow 0$.

Recently, a great attention has been devoted to the study of the following fractional magnetic Schr\"odinger equation
%a In the nonlocal framework, only few and recent works deal with fractional magnetic Schr\"odinger equations of the type
\begin{equation}\label{FMSE}
\e^{2s}(-\Delta)^{s}_{A}u+V(x)u=f(x, |u|^{2})u \quad \mbox{ in } \R^{N}.
\end{equation}
d'Avenia and Squassina \cite{DS} studied a class of minimization problems in the spirit of results due to Esteban and Lions in \cite{EL}.
%focused on the existence of ground state to \eqref{MSE} when $\e=1$, $V$ is constant and $f$ is a subcritical or critical nonlinearity. 
Fiscella et al. \cite{FPV} obtained the multiplicity of nontrivial solutions for a fractional magnetic problem in a bounded domain. 
%Zhang et al. \cite{ZSZ} considered a fractional magnetic Schr\"odinger equation with critical frequency and critical growth.
In \cite{AD} the author and d'Avenia dealt with the existence and multiplicity of solutions to \eqref{FMSE} for small $\e>0$ when $f$ has a subcritical growth and the potential $V$ satisfies the following global condition due to Rabinowitz \cite{Rab}:
\begin{equation}\label{RVC}
\liminf_{|x|\rightarrow \infty} V(x)>\inf_{x\in \R^{N}} V(x).
\end{equation}
In \cite{MPSZ} Mingqi et al. used suitable variational methods to prove the existence and multiplicity of nontrivial solutions for a class of super-and sub-linear fractional Schr\"odinger-Kirchhoff equations involving an external magnetic potential. We also mention \cite{Amjm, FVe, ZSZ} for other interesting results for nonlocal problems involving the operator $(-\Delta)^{s}_{A}$.

We stress that in the case $A\equiv 0$, equation \eqref{P}  becomes a fractional Kirchhoff equation of the type
\begin{equation}\label{FK}
\left(a\e^{2s}+b\e^{4s-3}[u]^{2}\right)(-\Delta)^{s}u+V(x)u=f(x,u) \quad \mbox{ in } \R^{3},
\end{equation}
which has been widely studied in the last decade. For instance, when $\e=1$, some existence and multiplicity results for fractional Kirchhoff equations in $\R^{N}$ can be found in \cite{AI1, PuSa, PXZ} and references therein; see also \cite{FMBS, FP, FV, Ny} for problems in bounded domains. In particular, in \cite{AI2, HZm} the authors studied fractional perturbed Kirchhoff-type problems, that is provided that $\e>0$ is sufficiently small.  
%and \cite{AI2, HZm} for perturbed Kirchhoff-type problems, that is $\e>0$ small enough.
%\cite{AFP, AI1, AI2, FP, FV, PXZ} and references therein.
% for results with $\e=1$ and \cite{AI2, HZm} when $\e>0$ is small enough. 
It is worthwhile to mention that Fiscella and Valdinoci \cite{FV}  proposed for the first time a stationary Kirchhoff model in the fractional setting, which  considers the nonlocal aspect of the tension arising from nonlocal measurements of the fractional length of the string.
Such model can be regarded as the nonlocal stationary analogue of the Kirchhoff equation
\begin{equation*}
%\label{Ke}
\rho u_{tt} - \left( \frac{p_{0}}{h}+ \frac{E}{2L}\int_{0}^{L} |u_{x}|^{2} dx \right) u_{xx} =0,
\end{equation*}
which was  presented by Kirchhoff \cite{Kir} in $1883$ as a generalization of the well-known D'Alembert's wave equation for free vibrations of elastic strings.  
The Kirchhoff's model takes into account the changes in length of the string produced by transverse vibrations. Here $u=u(x, t)$ is the transverse string displacement at the space coordinate $x$ and time $t$, $L$ is the length of the string, $h$ is the area of the cross section, $E$ is Young's modulus of the material, $\rho$ is the mass density, and $p_{0}$ is the initial tension; see \cite{B, Lions, P}.
In the classical framework,
%that is $s=1$ in \eqref{FK}, 
probably the first result concerning the following perturbed Kirchhoff equation 
\begin{equation}\label{CKE}
-\left(a\e^{2}+b\e \int_{\R^{3}}|\nabla u|^{2} dx\right)\Delta u+V(x)u=f(u) \quad \mbox{ in } \R^{3},
\end{equation}
has been obtained by He and Zou \cite{HZ}, who proved the multiplicity and concentration behavior of positive solutions to \eqref{CKE} for $\e>0$ small, under assumption \eqref{RVC} on $V$ and involving a subcritical nonlinearity. Subsequently, Wang et al. \cite{WTXZ} investigated the multiplicity and concentration phenomenon for \eqref{CKE} in presence of a critical term. Under local conditions $(V_1)$-$(V_2)$, Figueiredo and Santos Junior \cite{FJ} proved a multiplicity result for a subcritical Kirchhoff equation via the generalized Nehari manifold method. The existence and concentration of positive solutions for \eqref{CKE} with critical growth, has been considered in \cite{HLP}.

On the other hand, when we take $b=0$ in \eqref{FK}, then one has the following fractional Schr\"odinger equation (see \cite{Laskin})
\begin{equation}\label{FSE}
\e^{2s}(-\Delta)^{s}u+V(x)u=f(x,u) \quad \mbox{ in } \R^{N},
\end{equation} 
for which several existence and multiplicity results under different assumptions on $V$ and $f$ have been established via appropriate variational and topological methods; see \cite{A1, A3, DDPW, DMV, Secchi} and references therein. 
In particular way, Davila et al. \cite{DDPW} proved that if $V\in C^{1, \alpha}(\R^{N})\cap L^{\infty}(\R^{N})$ and $\inf_{x\in \R^{N}} V(x)>0$, then (\ref{P}) has multi-peak  solutions.
%Fall et al. \cite{FMV} established necessary and sufficient conditions on the smooth potential $V$ in order to produce concentration of solutions of (\ref{P}) as $\e\rightarrow 0$.
Alves and Miyagaki \cite{AM} (see  also \cite{A1, A3}) considered the existence and concentration of positive solutions of \eqref{FSE} when $V$ satisfies $(V_1)$-$(V_2)$ and $f$ has a subcritical growth. 
Recently, the author and Isernia \cite{AI3} studied the multiplicity and concentration of positive solutions for a fractional Schr\"odinger equation involving the fractional $p$-Laplacian operator when the potential satisfies \eqref{RVC} and the nonlinearity is assumed to be subcritical or critical.

Particularly motivated by the above works and the interest shared by the mathematical community on nonlocal magnetic problems,
%\cite{AFF, AM, AD, AI2, HZ}, 
in this paper we deal with the multiplicity and concentration of nontrivial solutions to \eqref{P} when $\e\rightarrow 0$, under assumptions $(V_1)$-$(V_2)$ and $(f_1)$-$(f_4)$.
More precisely, our main result is the following one:
\begin{thm}\label{thm1}
Assume that $(V_1)$-$(V_2)$ and $(f_1)$-$(f_4)$ hold. Then, for any $\delta>0$ such that
$$
M_{\delta}=\{x\in \R^{3}: dist(x, M)\leq \delta\}\subset \Lambda,
$$ 
there exists $\e_{\delta}>0$ such that, for any $\e\in (0, \e_{\delta})$, problem \eqref{P} has at least $cat_{M_{\delta}}(M)$ nontrivial solutions. Moreover, if $u_{\e}$ denotes one of these solutions and $x_{\e}$ is a global maximum point of $|u_{\e}|$, then we have 
$$
\lim_{\e\rightarrow 0} V(x_{\e})=V_{0}
$$	
and
$$
|u_{\e}(x)|\leq \frac{C\e^{3+2s}}{C\e^{3+2s}+|x-x_{\e}|^{3+2s}} \quad \forall x\in \R^{3}.
$$
\end{thm}
In what follows we give a sketch of the proof of Theorem \ref{thm1}. Firstly, 
using the change of variable $x\mapsto \e x$, instead of (\ref{P}), we can consider the following equivalent problem
\begin{equation}\label{Pe}
\left(a+b[u]_{A_\varepsilon}^{2}\right)(-\Delta)_{A_{\e}}^{s} u + V_{\e}( x)u=  f(|u|^{2})u  \quad \mbox{ in } \R^{3},
\end{equation}
where $A_{\e}(x):=A(\e x)$ and $V_{\e}(x):=V(\e x)$. Due to the lack of information on the behavior of $V$ at infinity, inspired by \cite{AFF, DF}, we modify the nonlinearity $f$ in an appropriate way, considering an auxiliary problem. In this way, we are able to apply suitable variational arguments to study the modified problem, and then we prove that, for $\e>0$ small enough, the solutions of the modified problem are also solutions of the original one.
More precisely, we fix $k>2$ and $a'>0$ such that $f(a')=\frac{V_{0}}{k}$, and we consider the function
$$
\hat{f}(t):=
\begin{cases}
f(t)& \text{ if $t \leq a'$} \\
\frac{V_{0}}{k}    & \text{ if $t >a'$}.
\end{cases}
$$ 
Let $t_{a'}, T_{a'}>0$ such that $t_{a'}<a'<T_{a'}$ and take $\xi\in C^{\infty}_{c}(\R, \R)$ such that:
\begin{compactenum}[$(\xi_1)$]
\item $\xi(t)\leq \hat{f}(t)$ for all $t\in [t_{a'}, T_{a'}]$,
\item $\xi(t_{a'})=\hat{f}(t_{a'})$, $\xi(T_{a'})=\hat{f}(T_{a'})$, $\xi'(t_{a'})=\hat{f}'(t_{a'})$ and $\xi'(T_{a'})=\hat{f}'(T_{a'})$, 
\item the map $t\mapsto \xi(t)$ is increasing for all $t\in [t_{a'}, T_{a'}]$.
\end{compactenum}
Then we define $\tilde{f}\in C^{1}(\R, \R)$ as follows:
$$
\tilde{f}(t):=
\begin{cases}
\hat{f}(t)& \text{ if $t\notin [t_{a'}, T_{a'}]$} \\
\xi(t)    & \text{ if $t\in [t_{a'}, T_{a'}]$}.
\end{cases}
$$ 
Finally, we introduce the following penalized nonlinearity $g: \R^{3}\times \R\rightarrow \R$ by setting
$$
g(x, t)=\chi_{\Lambda}(x)f(t)+(1-\chi_{\Lambda}(x))\tilde{f}(t),
$$
where $\chi_{\Lambda}$ is the characteristic function on $\Lambda$, and  we set $G(x, t)=\int_{0}^{t} g(x, \tau)\, d\tau$.
From assumptions $(f_1)$-$(f_4)$ and $(\xi_1)$-$(\xi_3)$, it follows that $g$ verifies the following properties:
\begin{compactenum}[($g_1$)]
\item $\displaystyle{\lim_{t\rightarrow 0} \frac{g(x, t)}{t}=0}$ uniformly in $x\in \R^{3}$;
\item $g(x, t)\leq f(t)$ for any $x\in \R^{3}$ and $t>0$;
\item $(i)$ $0< \frac{\theta}{2} G(x, t)\leq g(x, t)t$ for any $x\in \Lambda$ and $t>0$, \\
$(ii)$ $0\leq  G(x, t)\leq g(x, t)t\leq \frac{V(x)}{k}t$ and $0\leq g(x,t)\leq \frac{V(x)}{k}$ for any $x\in \Lambda^{c}$ and $t>0$;
\item $t\mapsto \frac{g(x,t)}{t}$ is increasing for all $x\in \Lambda$ and $t>0$.
\end{compactenum}
Then we introduce the following modified problem 
\begin{equation}\label{MPe}
\left(a+b[u]_{A_\varepsilon}^{2}\right)(-\Delta)_{A_{\e}}^{s} u + V_{\e}( x)u=  g_{\e}(x, |u|^{2})u \quad \mbox{ in } \R^{3}. 
\end{equation}
Let us note that if $u$ is a solution of (\ref{MPe}) such that 
\begin{equation}\label{ue}
|u(x)|\leq t_{a'} \mbox{ for all } x\in  \Lambda_{\e}^{c},
\end{equation}
where $\Lambda_{\e}:=\{x\in \R^{N}: \e x\in \Lambda\}$, then $u$ is also a solution of (\ref{Pe}).
Therefore, in order to study weak solutions of \eqref{MPe}, we look for critical points of the following functional associated with \eqref{Pe}:
%Therefore, our analysis consists in studying of critical points of the following functional associated with \eqref{Pe}:
$$
J_{\e}(u)=\frac{a}{2}[u]_{A_\varepsilon}^{2} + \frac{1}{2}\int_{\R^{3}}V_{\e}(x) |u|^{2}\, dx+\frac{b}{4}[u]_{A_\varepsilon}^{4}-\frac{1}{2}\int_{\R^{3}} G(\e x, |u|^{2})\, dx
$$
defined on the fractional magnetic Sobolev space 
$$
\h=\left\{u\in \mathcal{D}^{s}_{A_{\e}}(\R^{3}, \C): \int_{\R^{3}} V_{\e}(x) |u|^{2}\, dx<\infty\right\};
$$
see Section $2$. 
%We point out that a similar penalization approach has been also used in \cite{AM, A1, A3, AI2} when $A\equiv 0$.
% to deal with \eqref{FSE} ($A=b=0$ and $a=1$ in \eqref{P}), in \cite{AI2} to study \eqref{FK} ($A=b=0$ and $a=1$ in \eqref{P}) and \cite{HLP} 
%, that is $A=0$ and $a, b>0$ in \eqref{P}. 
%Anyway, the presence of the magnetic field does not permit to repeat the same arguments in our situation, and a more careful investigation will be needed to achieve some compactness results for $J_{\e}$. 
From $(g_1)$-$(g_3)$, it is easy to check that $J_{\e}$ has mountain pass geometry \cite{AR}. Anyway, the presence of the magnetic field and the lack of compactness of the embeddings $\h$ into $L^{p}(\R^{3}, \R)$, with $p\in (2, \2)$, create several difficulties to show that $J_{\e}$ verifies the Palais-Smale condition ($(PS)$ in short).
%Moreover, due to the Kirchhoff term $[u]^{2}_{A}(-\Delta)^{s}_{A}$, it is not clear that weak limits are critical points of $J_{\e}$. 
More precisely, the Kirchhoff term $[u]^{2}_{A}(-\Delta)^{s}_{A}$ does not permit to deduce in standard way that weak limits  of (bounded) Palais-Smale sequences of $J_{\e}$ are critical points of it. Therefore, a more careful investigation will be needed to recover some compactness property for the modified functional.
After that, combining some ideas introduced by Benci and Cerami \cite{BC}
%, which rely on suitable comparisons between the category of some sublevel sets of the modified functional and the category of the set $M$, 
with the Ljusternik-Schnirelman theory, we deduce a multiplicity result for the modified problem. We point out that the H\"older regularity assumption on the magnetic field $A$ and the fractional diamagnetic inequality \cite{DS}, will play a very important role to apply the minimax methods; see Sections $3$ and $4$.
In order to show that the solutions of \eqref{MPe} are indeed solutions of \eqref{Pe}, we need to show that \eqref{ue} holds for $\e$ small enough. 
This property will be proved using a suitable variant of the Moser iteration argument \cite{Moser} and a sort of Kato's inequality \cite{Kato} for $(-\Delta)^{s}_{A}$. We stress that $L^{\infty}$-estimates as in \cite{AFF} seem very hard to adapt in the nonlocal magnetic framework. Moreover, differently from the classical magnetic case (see \cite{CS, K}), we do not have a Kato's inequality for $(-\Delta)^{s}_{A}$ (except for $s=1/2$ as showed in \cite{HIL}, while here we are assuming $s>3/4$). Therefore, in this work we develop some new ingredients which we believe to be useful for future problems like \eqref{P}. We also provide a decay estimate of solutions of \eqref{P} which is in clear accordance with the results in \cite{FQT}.
As far as we know,  this is the first time that penalization methods jointly with Ljusternik-Schnirelmann theory are used to obtain multiple solutions for a fractional Kirchhoff equation with magnetic fields. 
\noindent

We organize the paper in the following way: in Section $2$ we collect some preliminary results for fractional Sobolev spaces; in Section $3$ we study the modified functional; in Section $4$ we provide a multiplicity result for \eqref{MPe}; finally, in Section $5$, we present the proof of Theorem \ref{thm1}.

\section{Preliminaries}\label{sec2}
In this preliminary section we fix the notations and we recall some technical results. 
We denote by $H^{s}(\R^{3}, \R)$ the fractional Sobolev space 
$$
H^{s}(\R^{3}, \R)=\{u\in L^{2}(\R^{3}, \R): [u]<\infty\},
$$
where
$$
[u]^{2}=\iint_{\R^{6}} \frac{|u(x)-u(y)|^{2}}{|x-y|^{3+2s}} dxdy
$$
is the Gagliardo seminorm.
We recall that the embedding $H^{s}(\R^{3}, \R)\subset L^{q}(\R^{3}, \R)$ is continuous for all $q\in [2, \2)$ and locally compact for all $q\in [1, \2)$; see \cite{DPV, MBRS}.\\
Let $L^{2}(\R^{3}, \C)$ be the space of complex-valued functions such that $\|u\|_{L^{2}(\R^{3})}^{2}=\int_{\R^{3}}|u|^{2}\, dx<\infty$ endowed with the inner product 
$\langle u, v\rangle_{L^{2}}=\Re\int_{\R^{3}} u\bar{v}\, dx$, where the bar denotes complex conjugation.

Let us denote by
$$
[u]^{2}_{A}:=\iint_{\R^{6}} \frac{|u(x)-e^{\imath (x-y)\cdot A(\frac{x+y}{2})} u(y)|^{2}}{|x-y|^{3+2s}} \, dxdy,
$$
and consider
$$
D_A^s(\R^3,\C)
:=
\left\{
u\in L^{2_s^*}(\R^3,\C) : [u]^{2}_{A}<\infty
\right\}.
$$
Then, we consider the Hilbert space
$$
H^{s}_{\e}:=
\left\{
u\in D_{A_{\e}}^s(\R^3,\C): \int_{\R^{3}} V_{\e}(x) |u|^{2}\, dx <\infty
\right\}
$$ 
endowed with the scalar product
\begin{align*}
&\langle u , v \rangle_{\e}:= a\Re\iint_{\R^{6}} \frac{(u(x)-e^{\imath(x-y)\cdot A_{\e}(\frac{x+y}{2})} u(y))\overline{(v(x)-e^{\imath(x-y)\cdot A_{\e}(\frac{x+y}{2})}v(y))}}{|x-y|^{3+2s}} dx dy+\Re \int_{\R^{3}} V_{\e}(x) u \bar{v} dx
\end{align*}
for all $u, v\in \h$, and let
$$
\|u\|_{\e}:=\sqrt{\langle u , u \rangle_{\e}}.
$$
In what follows we list some useful lemmas; see \cite{AD, DS} for more details.
\begin{lem}\cite{AD, DS}
The space $\h$ is complete and $C_c^\infty(\R^3,\C)$ is dense in $\h$. 
\end{lem}

\begin{lem}\label{DI}\cite{DS}
If $u\in H^{s}_{A}(\R^{3}, \C)$ then $|u|\in H^{s}(\R^{3}, \R)$ and we have
$$
[|u|]\leq [u]_{A}.
$$
\end{lem}

\begin{thm}\label{Sembedding}\cite{DS}
	The space $H^{s}_{\e}$ is continuously embedded in $L^{r}(\R^{3}, \C)$ for $r\in [2, 2^{*}_{s}]$, and compactly embedded in $L_{\rm loc}^{r}(\R^{3}, \C)$ for $r\in [1, 2^{*}_{s})$.\\
%Moreover, if $V_\infty=\infty$, then, for any bounded sequence $(u_{n})$  in $\h$, we have that, up to a subsequence, $(|u_{n}|)$ is strongly convergent in $L^{r}(\R^{3}, \R)$ for $r\in [2, 2^{*}_{s})$.
\end{thm}

\begin{lem}\label{aux}\cite{AD}
If $u\in H^{s}(\R^{3}, \R)$ and $u$ has compact support, then $w=e^{\imath A(0)\cdot x} u \in \h$.
\end{lem}

\noindent
We also recall a fractional version of Lions lemma whose proof can be found in \cite{FQT}:
\begin{lem}\label{Lions}\cite{FQT}
	Let $q\in [2, \2)$. If $(u_{n})$ is a bounded sequence in $H^{s}(\R^{3}, \R)$ and if 
	\begin{equation*}
	\lim_{n\rightarrow \infty} \sup_{y\in \R^{3}} \int_{B_{R}(y)} |u_{n}|^{q} dx=0
	\end{equation*}
	for some $R>0$, then $u_{n}\rightarrow 0$ in $L^{r}(\R^{3}, \R)$ for all $r\in (2, 2^{*}_{s})$.
\end{lem}

\section{variational setting and the modified functional}
%Using the change of variable $u(x)\mapsto u(\e x)$, we can see that (\ref{P}) is equivalent to the following problem
%\begin{equation}\label{R}
%(-\Delta)^{s}_{A_{\e}} u + V_{\e}(x)u+\phi_{|u|,t}u =  f(|u|^{2})u \mbox{ in } \R^{3},
%\end{equation}
%where $A_{\e}(x)=A(\e x)$ and $V_{\e}(x)=V(\e x)$. 
%This section is devoted to the proof of nontrivial solutions to the modified problem \eqref{Pe}.
Let us introduce the following functional $J_{\e}: \h\rightarrow \R$ defined as
\begin{align*}
J_{\e}(u)=\frac{1}{2}\|u\|_{\e}^{2}+\frac{b}{4}[u]_{A_{\e}}^{4}-\frac{1}{2}\int_{\R^{3}} G(\e x, |u|^{2})\, dx.
\end{align*}
%which is well-defined for any function $u$ belonging to the space
%$$
%\h=\left\{u\in \mathcal{D}^{s}_{A_{\e}}(\R^{3}, \C): \int_{\R^{3}} V_{\e}(x) |u|^{2}\, dx<\infty\right\}
%$$
%endowed with the norm 
%$$
%\|u\|^{2}_{\e}:=a[u]^{2}_{A_{\e}}+\|\sqrt{V_{\e}} |u|\|^{2}_{L^{2}(\R^{3})}.
%$$
It is easy to check that $J_{\e}\in C^{1}(\h, \R)$ and that its differential $J_{\e}'$ is given by
%, weak solutions to (\ref{MPe}) can be found as critical points of the Euler-Lagrange functional
%We note that the differential of $J_{\e}$ is given by
\begin{align*}
\langle J_{\e}'(u), v\rangle &=\langle u, v\rangle_{\e}+b[u]^{2}_{A_{\e}} \Re\iint_{\R^{6}} \frac{(u(x)-e^{\imath(x-y)\cdot A_{\e}(\frac{x+y}{2})} u(y))\overline{(v(x)-e^{\imath(x-y)\cdot A_{\e}(\frac{x+y}{2})}v(y))}}{|x-y|^{3+2s}} dx dy \\
&-\Re\int_{\R^{3}} g(\e x, |u|^{2})u\bar{v} dx.
\end{align*}
Therefore, weak solutions to (\ref{MPe}) can be found as critical points of $J_{\e}$.
We will also consider the following family of autonomous problems associated to \eqref{MPe}, that is for all $\mu>0$
\begin{equation}\label{AP0}
(a+b[u]^{2})(-\Delta)^{s} u + \mu u=  f(u^{2})u \mbox{ in } \R^{3}, 
\end{equation}
and we introduce the  corresponding energy functional $J_{\mu}: H^{s}_{\mu}\rightarrow \R$ given by
\begin{align*}
J_{\mu}(u)
%&=\frac{1}{2}\int_{\R^{N}}a |(-\Delta)^{\frac{s}{2}}u|^{2}+\mu |u|^{2}\, dx+\frac{b}{4}\left(\int_{\R^{3}}|(-\Delta)^{\frac{s}{2}}u|^{2} dx\right)^{2}-\frac{1}{2}\int_{\R^{3}} F(u^{2})\, dx\\
=\frac{1}{2}\|u\|^{2}_{\mu}+\frac{b}{4}[u]^{4}-\frac{1}{2}\int_{\R^{3}} F(u^{2})\, dx
\end{align*}
where $H^{s}_{\mu}$ stands for the fractional Sobolev space $H^{s}(\R^{3}, \R)$ endowed with the norm 
$$\|u\|_{\mu}^{2}=a[u]^{2}+\mu\|u\|^{2}_{L^{2}(\R^{3})}.$$ 

\noindent
We stress that, under the assumptions on $f$, $J_{\e}$ possesses a mountain pass geometry \cite{AR}. Indeed, we can prove that:
\begin{lem}\label{MPG}
\begin{compactenum}[$(i)$]
\item $J_{\e}(0)=0$;
\item there exists $\alpha, \rho>0$ such that $J_{\e}(u)\geq \alpha$ for any $u\in \h$ such that $\|u\|_{\e}=\rho$;
\item there exists $e\in \h$ with $\|e\|_{\e}>\rho$ such that $J_{\e}(e)<0$.
\end{compactenum}
\end{lem}
\begin{proof}
By $(g_1)$ and $(g_2)$, for all $\delta>0$ there exists $C_{\delta}>0$ such that 
$$
|G(\e x, t^2)|\leq \delta |t|^{4}+C_{\delta}|t|^{q} \mbox{ for all } x\in \R^{3}, t\in \R.
$$
This fact combined with Theorem \ref{Sembedding} implies that
$$
J_{\e}(u)\geq C\|u\|^{2}_{\e}-\delta C\|u\|_{\e}^{4}-C_{\delta} \|u\|^{q}_{\e}.
$$
Since $q\in (4, \2)$, it follows that $(i)$ holds.
Now, fix $u\in \h\setminus\{0\}$ with $supp(u)\subset \Lambda_{\e}$. By $(f_3)$ we get
%Regarding $(ii)$, we can see that $(f_3)$, we have for any $u\in \h\setminus\{0\}$ with $supp(u)\subset \Lambda_{\e}$ and $T>1$
\begin{align*}
J_{\e}(Tu)&\leq \frac{T^{2}}{2} \|u\|^{2}_{\e}+b\frac{T^{4}}{4}[u]_{A_{\e}}^{4}-\frac{1}{2}\int_{\Lambda_{\e}} F(T^{2}|u|^{2})\, dx \\
&\leq \frac{T^{2}}{2}\|u\|^{2}_{\e}+\frac{T^{4}}{4}b[u]_{A_{\e}}^{4}-CT^{\theta} \int_{\Lambda_{\e}} |u|^{\theta}\, dx+C
\end{align*}
which in view of $\theta>4$ yields $J_{\e}(Tu)\rightarrow -\infty$ as $T\rightarrow \infty$.
\end{proof}

\noindent
From Lemma \ref{MPG} it follows that we can define the minimax level 
\begin{align*}
c_{\e}=\inf_{\gamma\in \Gamma_{\e}} \max_{t\in [0, 1]} J_{\e}(\gamma(t)) \quad \mbox{ where } \quad \Gamma_{\e}=\{\gamma\in C([0, 1], \h): \gamma(0)=0 \mbox{ and } J_{\e}(\gamma(1))<0\}.
\end{align*}
Using a version of the mountain pass theorem without $(PS)$ condition (see \cite{W}), we can find a Palais-Smale sequence $(u_{n})$ at the level $c_{\e}$.
Now, we prove that $J_{\e}$ enjoys of the following compactness property:
\begin{lem}\label{PSc}
Let $c\in \R$. Then $J_{\e}$ satisfies the Palais-Smale condition at the level $c$.
\end{lem}
\begin{proof}
Let $(u_{n})\subset \h$ be a $(PS)_{c}$-sequence of $J_{\e}$, that is 
$
J_{\e}(u_{n})\rightarrow c \mbox{ and } J_{\e}'(u_{n})\rightarrow 0
$
as $n\rightarrow \infty$.
Since the proof is very long, we divide it into four steps.\\
{\bf Step 1} The sequence $(u_{n})$ is bounded in $\h$. 
Indeed, by $(g_3)$ we get
\begin{align*}
c+o_{n}(1)\|u_{n}\|_{\e}&= J_{\e}(u_{n})-\frac{1}{\theta}\langle J'_{\e}(u_{n}), u_{n}\rangle \\
&= \left(\frac{1}{2}-\frac{1}{\theta}\right)\|u_{n}\|^{2}_{\e}+\left( \frac{1}{4}-\frac{1}{\theta}\right)b[u_{n}]_{A_{\e}}^{4}\\
&+\frac{1}{\theta}\int_{\R^{3}} \left[g_{\e}(x, |u_{n}|^{2})|u_{n}|^{2}-\frac{\theta}{2} G_{\e}(x, |u_{n}|^{2})\right]\, dx\\
&\geq \left(\frac{1}{2}-\frac{1}{\theta}\right)\|u_{n}\|^{2}_{\e}
%+\frac{1}{\theta}\int_{\Lambda_{\e}} \left(g_{\e}(x, |u_{n}|^{2})|u_{n}|^{2}-\frac{\theta}{2}G(\e x, |u_{n}|^{2}) \right)dx  \\
+\left(\frac{2-\theta}{2\theta} \right)\int_{\Lambda^{c}_{\e}} G_{\e}(x, |u_{n}|^{2})\, dx \\
%-\frac{1}{2} \int_{\R^{3}}\frac{V(\e x)}{k}|u_{n}|^{2} dx \\
&\geq \left(\frac{1}{2}-\frac{1}{\theta}\right)\|u_{n}\|^{2}_{\e}+\left(\frac{2-\theta}{2\theta k} \right)\int_{\Lambda^{c}_{\e}} V(\e x) |u_{n}|^{2}\, dx \\
&\geq \left(\frac{\theta-2}{2\theta}\right)\left( 1-\frac{1}{k} \right)\|u_{n}\|^{2}_{\e},
\end{align*}
and using the fact that $k>2$, we can conclude that $(u_{n})$ is bounded in $\h$. 
Consequently, we may assume that $u_{n}\rightharpoonup u$ in $\h$ and $[u_{n}]_{A_{\e}}^{2}\rightarrow \ell^{2}\in (0, \infty)$. \\
Set 
$$
M_{n}:=a+b[u_{n}]_{A_{\e}}^{2},
$$
and we note that $M_{n}\rightarrow a+b\ell^{2}$ as $n\rightarrow \infty$.
In what follows we prove that $u_{n}\rightarrow u$ in $\h$. \\
{\bf Step 2} For any $\xi>0$ there exists $R=R_{\xi}>0$ such that $\Lambda_{\e}\subset B_{R}$ and
\begin{equation}\label{T}
\limsup_{n\rightarrow \infty}\int_{B_{R}^{c}} \int_{\R^{3}}a \frac{|u_{n}(x)-u_{n}(y)e^{\imath A_{\e}(\frac{x+y}{2})\cdot (x-y)}|^{2}}{|x-y|^{3+2s}} dx dy+\int_{B_{R}^{c}} V_{\e}(x)|u_{n}(x)|^{2}\, dx\leq \xi.
\end{equation}
Let $\eta_{R}\in C^{\infty}(\R^{3}, \R)$ be such that $0\leq \eta_{R}\leq 1$, $\eta_{R}=0$ in $B_{\frac{R}{2}}$, $\eta_{R}=1$ in $B_{R}^{c}$ and $|\nabla \eta_{R}|\leq \frac{C}{R}$ for some $C>0$ independent of $R$.
From $\langle J'_{\e}(u_{n}), \eta_{R}u_{n}\rangle =o_{n}(1)$ it follows that
\begin{align*}
&\Re \left[M_{n} \iint_{\R^{6}} \frac{(u_{n}(x)-u_{n}(y)e^{\imath A_{\e}(\frac{x+y}{2})\cdot (x-y)})\overline{(u_{n}(x)\eta_{R}(x)-u_{n}(y)\eta_{R}(y)e^{\imath A_{\e}(\frac{x+y}{2})\cdot (x-y)})}}{|x-y|^{3+2s}}\, dx dy \right]\\
&+\int_{\R^{3}} V_{\e}\eta_{R} |u_{n}|^{2}\, dx=\int_{\R^{N}} g_{\e}(x, |u_{n}|^{2})|u_{n}|^{2}\eta_{R}\, dx+o_{n}(1).
\end{align*}
Since
\begin{align*}
&\Re\left(\iint_{\R^{6}} \frac{(u_{n}(x)-u_{n}(y)e^{\imath A_{\e}(\frac{x+y}{2})\cdot (x-y)})\overline{(u_{n}(x)\eta_{R}(x)-u_{n}(y)\eta_{R}(y)e^{\imath A_{\e}(\frac{x+y}{2})\cdot (x-y)})}}{|x-y|^{3+2s}}\, dx dy \right)\\
&=\Re\left(\iint_{\R^{6}} \overline{u_{n}(y)}e^{-\imath A_{\e}(\frac{x+y}{2})\cdot (x-y)}\frac{(u_{n}(x)-u_{n}(y)e^{\imath A_{\e}(\frac{x+y}{2})\cdot (x-y)})(\eta_{R}(x)-\eta_{R}(y))}{|x-y|^{3+2s}}  \,dx dy\right)\\
&+\iint_{\R^{6}} \eta_{R}(x)\frac{|u_{n}(x)-u_{n}(y)e^{\imath A_{\e}(\frac{x+y}{2})\cdot (x-y)}|^{2}}{|x-y|^{3+2s}}\, dx dy,
\end{align*}
and $M_{n}\geq a$, we can use $(g_3)$-(ii) to get
\begin{align}\label{PS1}
&a\iint_{\R^{6}} \eta_{R}(x)\frac{|u_{n}(x)-u_{n}(y)e^{\imath A_{\e}(\frac{x+y}{2})\cdot (x-y)}|^{2}}{|x-y|^{3+2s}}\, dx dy+\int_{\R^{3}} V_{\e}(x)\eta_{R} |u_{n}|^{2}\, dx\nonumber\\
&\leq -\Re \left[M_{n} \iint_{\R^{6}} \overline{u_{n}(y)}e^{-\imath A_{\e}(\frac{x+y}{2})\cdot (x-y)}\frac{(u_{n}(x)-u_{n}(y)e^{\imath A_{\e}(\frac{x+y}{2})\cdot (x-y)})(\eta_{R}(x)-\eta_{R}(y))}{|x-y|^{3+2s}}  \,dx dy\right] \nonumber\\
&+\frac{1}{k}\int_{\R^{3}} V_{\e} \eta_{R} |u_{n}|^{2}\, dx+o_{n}(1).
\end{align}
Using H\"older inequality and the boundedness of $(u_{n})$ in $\h$ we obtain
\begin{align}\label{PS2}
&\left|\Re \left[M_{n}\iint_{\R^{6}} \overline{u_{n}(y)}e^{-\imath A_{\e}(\frac{x+y}{2})\cdot (x-y)}\frac{(u_{n}(x)-u_{n}(y)e^{\imath A_{\e}(\frac{x+y}{2})\cdot (x-y)})(\eta_{R}(x)-\eta_{R}(y))}{|x-y|^{3+2s}}  \,dx dy\right]\right| \nonumber\\
&\leq C\left(\iint_{\R^{6}} \frac{|u_{n}(x)-u_{n}(y)e^{\imath A_{\e}(\frac{x+y}{2})\cdot (x-y)}|^{2}}{|x-y|^{3+2s}}\,dxdy  \right)^{\frac{1}{2}} \left(\iint_{\R^{6}} |u_{n}(y)|^{2}\frac{|\eta_{R}(x)-\eta_{R}(y)|^{2}}{|x-y|^{3+2s}} \, dxdy\right)^{\frac{1}{2}} \nonumber\\
&\leq C \left(\iint_{\R^{6}} |u_{n}(y)|^{2}\frac{|\eta_{R}(x)-\eta_{R}(y)|^{2}}{|x-y|^{3+2s}} \, dxdy\right)^{\frac{1}{2}}.
\end{align}
Now, we show that
%Taking into account (see Lemma $2.1$ in \cite{A3} or Lemma $3.4$ in \cite{AI2}) 
\begin{equation}\label{PS3}
\limsup_{R\rightarrow \infty}\limsup_{n\rightarrow \infty} \iint_{\R^{6}} |u_{n}(y)|^{2}\frac{|\eta_{R}(x)-\eta_{R}(y)|^{2}}{|x-y|^{3+2s}} \, dxdy=0.
\end{equation}
Let us note that  
$$
\R^{6}=((\R^{3}\setminus B_{2R})\times (\R^{3}\setminus B_{2R})) \cup ((\R^{3}\setminus B_{2R})\times B_{2R})\cup (B_{2R}\times \R^{3})=: X^{1}_{R}\cup X^{2}_{R} \cup X^{3}_{R}.
$$
Therefore
\begin{align}\label{Pa1}
&\iint_{\R^{6}}\frac{|\eta_{R}(x)-\eta_{R}(y)|^{2}}{|x-y|^{3+2s}} |u_{n}(x)|^{2} dx dy =\iint_{X^{1}_{R}}\frac{|\eta_{R}(x)-\eta_{R}(y)|^{2}}{|x-y|^{3+2s}} |u_{n}(x)|^{2} dx dy \nonumber \\
&+\iint_{X^{2}_{R}}\frac{|\eta_{R}(x)-\eta_{R}(y)|^{2}}{|x-y|^{3+2s}} |u_{n}(x)|^{2} dx dy+
\iint_{X^{3}_{R}}\frac{|\eta_{R}(x)-\eta_{R}(y)|^{2}}{|x-y|^{3+2s}} |u_{n}(x)|^{2} dx dy.
\end{align}
%Now, we estimate each integral in (\ref{Pa1}).
Since $\eta_{R}=1$ in $\R^{3}\setminus B_{2R}$, we can see that
\begin{align}\label{Pa2}
\iint_{X^{1}_{R}}\frac{|u_{n}(x)|^{2}|\eta_{R}(x)-\eta_{R}(y)|^{2}}{|x-y|^{3+2s}} dx dy=0.
\end{align}
Now, fix $K>4$, and we observe that
\begin{equation*}
X^{2}_{R}=(\R^{3} \setminus B_{2R})\times B_{2R} \subset ((\R^{3}\setminus B_{KR})\times B_{2R})\cup ((B_{KR}\setminus B_{2R})\times B_{2R}) 
\end{equation*}
If $(x, y) \in (\R^{3}\setminus B_{KR})\times B_{2R}$, then
\begin{equation*}
|x-y|\geq |x|-|y|\geq |x|-2R>\frac{|x|}{2}. 
\end{equation*}
Therefore, using $0\leq \eta_{R}\leq 1$, $|\nabla \eta_{R}|\leq \frac{C}{R}$ and applying H\"older inequality we obtain
\begin{align}\label{Pa3}
&\iint_{X^{2}_{R}}\frac{|u_{n}(x)|^{2}|\eta_{R}(x)-\eta_{R}(y)|^{2}}{|x-y|^{3+2s}} dx dy \nonumber \\
&=\int_{\R^{3}\setminus B_{KR}} \int_{B_{2R}} \frac{|u_{n}(x)|^{2}|\eta_{R}(x)-\eta_{R}(y)|^{2}}{|x-y|^{3+2s}} dx dy + \int_{B_{kR}\setminus B_{2R}} \int_{B_{2R}} \frac{|u_{n}(x)|^{2}|\eta_{R}(x)-\eta_{R}(y)|^{2}}{|x-y|^{3+2s}} dx dy \nonumber \\
&\leq 2^{2+3+2s} \int_{\R^{3}\setminus B_{KR}} \int_{B_{2R}} \frac{|u_{n}(x)|^{2}}{|x|^{3+2s}}\, dxdy+ \frac{C}{R^{2}} \int_{B_{KR}\setminus B_{2R}} \int_{B_{2R}} \frac{|u_{n}(x)|^{2}}{|x-y|^{3+2(s-1)}}\, dxdy \nonumber \\
&\leq CR^{3} \int_{\R^{3}\setminus B_{KR}} \frac{|u_{n}(x)|^{2}}{|x|^{3+2s}}\, dx + \frac{C}{R^{2}} (KR)^{2(1-s)} \int_{B_{KR}\setminus B_{2R}} |u_{n}(x)|^{2} dx \nonumber \\
&\leq CR^{3} \left( \int_{\R^{3}\setminus B_{KR}} |u_{n}(x)|^{2^{*}_{s}} dx \right)^{\frac{2}{2^{*}_{s}}} \left(\int_{\R^{3}\setminus B_{KR}}\frac{1}{|x|^{\frac{3^{2}}{2s} +3}}\, dx \right)^{\frac{2s}{3}} + \frac{C K^{2(1-s)}}{R^{2s}} \int_{B_{KR}\setminus B_{2R}} |u_{n}(x)|^{2} dx \nonumber \\
&\leq \frac{C}{K^{3}} \left( \int_{\R^{3}\setminus B_{KR}} |u_{n}(x)|^{2^{*}_{s}} dx \right)^{\frac{2}{2^{*}_{s}}} + \frac{C K^{2(1-s)}}{R^{2s}} \int_{B_{KR}\setminus B_{2R}} |u_{n}(x)|^{2} dx \nonumber \\
&\leq \frac{C}{K^{3}}+ \frac{C K^{2(1-s)}}{R^{2s}} \int_{B_{KR}\setminus B_{2R}} |u_{n}(x)|^{2} dx.
\end{align}
Take $\delta\in (0,1)$, and we obtain
\begin{align}\label{Ter1}
&\iint_{X^{3}_{R}} \frac{|u_{n}(x)|^{2} |\eta_{R}(x)- \eta_{R}(y)|^{2}}{|x-y|^{3+2s}}\, dxdy \nonumber\\
&\leq \int_{B_{2R}\setminus B_{\delta R}} \int_{\R^{3}} \frac{|u_{n}(x)|^{2} |\eta_{R}(x)- \eta_{R}(y)|^{2}}{|x-y|^{3+2s}}\, dxdy + \int_{B_{\delta R}} \int_{\R^{3}} \frac{|u_{n}(x)|^{2} |\eta_{R}(x)- \eta_{R}(y)|^{2}}{|x-y|^{3+2s}}\, dxdy. 
\end{align}
%Let us estimate the first integral in \eqref{Ter1}. Then, 
Since
\begin{align*}
\int_{B_{2R}\setminus B_{\delta R}} \int_{\R^{3} \cap \{y: |x-y|<R\}} \frac{|u_{n}(x)|^{2} |\eta_{R}(x)- \eta_{R}(y)|^{2}}{|x-y|^{3+2s}}\, dxdy \leq \frac{C}{R^{2s}} \int_{B_{2R}\setminus B_{\delta R}} |u_{n}(x)|^{2} dx
\end{align*}
and 
\begin{align*}
\int_{B_{2R}\setminus B_{\delta R}} \int_{\R^{3} \cap \{y: |x-y|\geq R\}} \frac{|u_{n}(x)|^{2} |\eta_{R}(x)- \eta_{R}(y)|^{2}}{|x-y|^{3+2s}}\, dxdy \leq \frac{C}{R^{2s}} \int_{B_{2R}\setminus B_{\delta R}} |u_{n}(x)|^{2} dx,
\end{align*}
we can see that
%from which we have
\begin{align}\label{Ter2}
\int_{B_{2R}\setminus B_{\delta R}} \int_{\R^{3}} \frac{|u_{n}(x)|^{2} |\eta_{R}(x)- \eta_{R}(y)|^{2}}{|x-y|^{3+2s}}\, dxdy \leq \frac{C}{R^{2s}} \int_{B_{2R}\setminus B_{\delta R}} |u_{n}(x)|^{2} dx. 
\end{align}
On the other hand, from the definition of $\eta_{R}$, $\e\in (0,1)$, and $\eta_{R}\leq 1$ we obtain
%Now, by using the definition of $\eta_{R}$, $\e\in (0,1)$, and $\eta_{R}\leq 1$, we have 
\begin{align}\label{Ter3}
\int_{B_{\delta R}} \int_{\R^{3}} \frac{|u_{n}(x)|^{2} |\eta_{R}(x)- \eta_{R}(y)|^{2}}{|x-y|^{3+2s}}\, dxdy &= \int_{B_{\delta R}} \int_{\R^{3}\setminus B_{R}} \frac{|u_{n}(x)|^{2} |\eta_{R}(x)- \eta_{R}(y)|^{2}}{|x-y|^{3+2s}}\, dxdy\nonumber \\
&\leq 4 \int_{B_{\delta R}} \int_{\R^{3}\setminus B_{R}} \frac{|u_{n}(x)|^{2}}{|x-y|^{3+2s}}\, dxdy\nonumber \\
&\leq C \int_{B_{\delta R}} |u_{n}|^{2} dx \int_{(1-\delta)R}^{\infty} \frac{1}{r^{1+2s}} dr\nonumber \\
&=\frac{C}{[(1-\delta)R]^{2s}} \int_{B_{\delta R}} |u_{n}|^{2} dx
\end{align}
where we used the fact that if $(x, y) \in B_{\delta R}\times (\R^{3} \setminus B_{R})$, then $|x-y|>(1-\delta)R$. \\
Then \eqref{Ter1}, \eqref{Ter2} and \eqref{Ter3} yield
%Taking into account \eqref{Ter1}, \eqref{Ter2} and \eqref{Ter3} we deduce 
\begin{align}\label{Pa4}
\iint_{X^{3}_{R}} &\frac{|u_{n}(x)|^{2} |\eta_{R}(x)- \eta_{R}(y)|^{2}}{|x-y|^{3+2s}}\, dxdy \nonumber\\
&\leq \frac{C}{R^{2s}} \int_{B_{2R}\setminus B_{\delta R}} |u_{n}(x)|^{2} dx + \frac{C}{[(1-\e)R]^{2s}} \int_{B_{\delta R}} |u_{n}(x)|^{2} dx. 
\end{align}
In view of \eqref{Pa1}, \eqref{Pa2}, \eqref{Pa3} and \eqref{Pa4} we can infer 
%Putting together \eqref{Pa1}, \eqref{Pa2}, \eqref{Pa3} and \eqref{Pa4}, we can infer 
\begin{align}\label{Pa5}
\iint_{\R^{6}} &\frac{|u_{n}(x)|^{2} |\eta_{R}(x)- \eta_{R}(y)|^{2}}{|x-y|^{3+2s}}\, dxdy \nonumber \\
&\leq \frac{C}{K^{3}} + \frac{CK^{2(1-s)}}{R^{2s}} \int_{B_{KR}\setminus B_{2R}} |u_{n}(x)|^{2} dx + \frac{C}{R^{2s}} \int_{B_{2R}\setminus B_{\delta R}} |u_{n}(x)|^{2} dx \nonumber \\
&\quad+ \frac{C}{[(1-\delta)R]^{2s}}\int_{B_{\delta R}} |u_{n}(x)|^{2} dx. 
\end{align}
Since $(|u_{n}|)$ is bounded in $H^{s}(\R^{3}, \R)$, using Sobolev embedding $H^{s}(\R^{3}, \R)\subset L^{\2}(\R^{3}, \R)$ (see \cite{DPV}), we may assume that $|u_{n}|\rightarrow u$ in $L^{2}_{loc}(\R^{3}, \R)$ for some $u\in H^{s}(\R^{3}, \R)$. Letting the limit as $n\rightarrow \infty$ in \eqref{Pa5} we find
\begin{align*}
&\limsup_{n\rightarrow \infty} \iint_{\R^{6}} \frac{|u_{n}(x)|^{2} |\eta_{R}(x)- \eta_{R}(y)|^{2}}{|x-y|^{3+2s}}\, dxdy\\
&\leq \frac{C}{K^{3}} + \frac{CK^{2(1-s)}}{R^{2s}} \int_{B_{KR}\setminus B_{2R}} |u(x)|^{2} dx + \frac{C}{R^{2s}} \int_{B_{2R}\setminus B_{\delta R}} |u(x)|^{2} dx + \frac{C}{[(1-\delta)R]^{2s}}\int_{B_{\delta R}} |u(x)|^{2} dx \\
&\leq \frac{C}{K^{3}} + CK^{2} \left( \int_{B_{KR}\setminus B_{2R}} |u(x)|^{2^{*}_{s}} dx\right)^{\frac{2}{2^{*}_{s}}} + C\left(\int_{B_{2R}\setminus B_{\delta R}} |u(x)|^{2^{*}_{s}} dx\right)^{\frac{2}{2^{*}_{s}}} \nonumber \\
&\quad+ C\left( \frac{\delta}{1-\delta}\right)^{2s} \left(\int_{B_{\delta R}} |u(x)|^{2^{*}_{s}} dx\right)^{\frac{2}{2^{*}_{s}}}, 
\end{align*}
where in the last passage we used H\"older inequality. 
Since $u\in L^{2^{*}_{s}}(\R^{3}, \R)$, $K>4$ and $\delta \in (0,1)$ we can see that
\begin{align*}
\limsup_{R\rightarrow \infty} \int_{B_{KR}\setminus B_{2R}} |u(x)|^{2^{*}_{s}} dx = \limsup_{R\rightarrow \infty} \int_{B_{2R}\setminus B_{\delta R}} |u(x)|^{2^{*}_{s}} dx = 0.
\end{align*}
Thus, taking $\delta= \frac{1}{K}$, we have
\begin{align*}
&\limsup_{R\rightarrow \infty} \limsup_{n\rightarrow \infty} \iint_{\R^{6}} \frac{|u_{n}(x)|^{2} |\eta_{R}(x)- \eta_{R}(y)|^{2}}{|x-y|^{3+2s}}\, dxdy\\
&\leq \lim_{K\rightarrow \infty} \limsup_{R\rightarrow \infty} \Bigl[\, \frac{C}{K^{3}} + CK^{2} \left( \int_{B_{KR}\setminus B_{2R}} |u(x)|^{2^{*}_{s}} dx\right)^{\frac{2}{2^{*}_{s}}} + C\left(\int_{B_{2R}\setminus B_{\frac{1}{K} R}} |u(x)|^{2^{*}_{s}} dx\right)^{\frac{2}{2^{*}_{s}}} \\
&+ C\left(\frac{1}{K-1}\right)^{2s} \left(\int_{B_{\frac{1}{K} R}} |u(x)|^{2^{*}_{s}} dx\right)^{\frac{2}{2^{*}_{s}}}\, \Bigr]\\
&\leq \lim_{K\rightarrow \infty} \frac{C}{K^{3}} + C\left(\frac{1}{K-1}\right)^{2s} \left(\int_{\R^{3}} |u(x)|^{2^{*}_{s}} dx \right)^{\frac{2}{2^{*}_{s}}}= 0,
\end{align*}
which implies that \eqref{PS3} holds true.  Hence, 
putting together \eqref{PS1}, \eqref{PS2} and \eqref{PS3} we can deduce that
$$
\limsup_{R\rightarrow \infty}\limsup_{n\rightarrow \infty} \int_{B_{R}^{c}} \int_{\R^{3}}a \frac{|u_{n}(x)-u_{n}(y)e^{\imath A_{\e}(\frac{x+y}{2})\cdot (x-y)}|^{2}}{|x-y|^{3+2s}} dx dy+\int_{B_{R}^{c}} V_{\e}(x)|u_{n}(x)|^{2}\, dx=0,
$$
%we can deduce that \eqref{PS1}, \eqref{PS2} and \eqref{PS3} yield
%$$
%\limsup_{R\rightarrow \infty}\limsup_{n\rightarrow \infty} \int_{B_{R}^{c}} a |(-\Delta)^{\frac{s}{2}}_{A_{\e}}u_{n}|^{2}+V_{\e} |u_{n}|^{2}\, dx=0.
%$$
which yields \eqref{T}. \\
{\bf Step 3}  For all $R>0$ it holds
\begin{align}\label{SPLITbr}
\lim_{n\rightarrow \infty} &\int_{B_{R}} dx \int_{\R^{3}} \frac{|u_{n}(x)-u_{n}(y)e^{\imath A_{\e}(\frac{x+y}{2})\cdot (x-y)}|^{2}}{|x-y|^{3+2s}} \, dy + \int_{B_{R}} V_{\e} |u_{n}|^{2} dx \nonumber\\
&= \int_{B_{R}} dx \int_{\R^{3}} \frac{|u(x)-u(y)e^{\imath A_{\e}(\frac{x+y}{2})\cdot (x-y)}|^{2}}{|x-y|^{3+2s}} \, dy + \int_{B_{R}} V_{\e} |u|^{2} dx.  
\end{align}
%Set $M_{n}:=a+b[u_{n}]_{A_{\e}}^{2}$.
Let $\eta_{\rho}\in C^{\infty}(\R^{3}, \R)$ be such that $\eta_{\rho}=1$ in $B_{\rho}$ and $\eta_{\rho}=0$ in $B_{2\rho}^{c}$, with $0\leq \eta_{\rho}\leq 1$. 

Set
\begin{equation*}
\Phi_{n}(x):=M_{n}\int_{\R^{3}} \frac{|(u_{n}(x)-u(x)) - (u_{n}(y)- u(y))e^{\imath A_{\e}(\frac{x+y}{2})\cdot (x-y)})|^{2}}{|x-y|^{3+2s}} \, dy+ V_{\e} |u_{n}(x)-u(x)|^{2}. 
\end{equation*}
Fix $R>0$ and choose $\rho>R$. Then we have 
\begin{align}\label{Phi}
0&\leq \int_{B_{R}} \Phi_{n}(x) \, dx= \int_{B_{R}} \Phi_{n}(x)\eta_{\rho}(x) \,dx \nonumber \\
&\leq M_{n} \iint_{\R^{6}} \frac{|(u_{n}(x) - u_{n}(y)e^{\imath A_{\e}(\frac{x+y}{2})\cdot (x-y)})- (u(x)- u(y)e^{\imath A_{\e}(\frac{x+y}{2})\cdot (x-y)})|^{2}}{|x-y|^{3+2s}} \eta_{\rho}(x)\, dxdy\nonumber \\
&\quad+ \int_{\R^{3}} V_{\e} |u_{n}- u|^{2}\eta_{\rho} \,dx \nonumber \\
&= M_{n}\iint_{\R^{6}} \frac{|u_{n}(x) - u_{n}(y)e^{\imath A_{\e}(\frac{x+y}{2})\cdot (x-y)}|^{2}}{|x-y|^{3+2s}} \eta_{\rho}(x)\, dxdy + \int_{\R^{3}} V_{\e} |u_{n}|^{2}\eta_{\rho}\,dx  \nonumber \\
&\quad + M_{n}\iint_{\R^{6}} \frac{|u(x) - u(y)e^{\imath A_{\e}(\frac{x+y}{2})\cdot (x-y)}|^{2}}{|x-y|^{3+2s}} \eta_{\rho}(x)\, dxdy + \int_{\R^{3}} V_{\e} |u|^{2}\eta_{\rho} \,dx \nonumber \\
&\quad -2\Re\left[ M_{n}\iint_{\R^{6}} \frac{(u_{n}(x) - u_{n}(y)e^{\imath A_{\e}(\frac{x+y}{2})\cdot (x-y)})\overline{(u(x)-u(y)e^{\imath A_{\e}(\frac{x+y}{2})\cdot (x-y)})}}{|x-y|^{3+2s}} \eta_{\rho}(x)\, dxdy \right. \nonumber \\
&\quad \left.+ \int_{\R^{3}} V_{\e} u_{n}\bar{u}\,\eta_{\rho} \,dx\right] \nonumber \\
&=I_{n,\rho}-II_{n,\rho}+III_{n,\rho}+IV_{n,\rho}\leq |I_{n,\rho}|+ |II_{n,\rho}|+ |III_{n,\rho}|+ |IV_{n,\rho}|,
\end{align}
where 
\begin{align*}
&I_{n,\rho}:=M_{n}\iint_{\R^{6}} \frac{|u_{n}(x) - u_{n}(y)e^{\imath A_{\e}(\frac{x+y}{2})\cdot (x-y)}|^{2}}{|x-y|^{3+2s}} \eta_{\rho}(x)\, dxdy + \int_{\R^{3}} V_{\e} |u_{n}|^{2}\eta_{\rho}\,dx- \int_{\R^{3}} g(\e x, |u_{n}|^{2}) |u_{n}|^{2} \eta_{\rho}\, dx,\\
&II_{n,\rho}:= \Re\left[ M_{n}\iint_{\R^{6}} \frac{(u_{n}(x) - u_{n}(y)e^{\imath A_{\e}(\frac{x+y}{2})\cdot (x-y)}) \overline{(u(x)-u(y)e^{\imath A_{\e}(\frac{x+y}{2})\cdot (x-y)})}}{|x-y|^{3+2s}} \eta_{\rho}(x)\, dxdy \right. \\
&\qquad\left.+ \int_{\R^{3}} V_{\e} u_{n}\bar{u} \eta_{\rho}\,dx\right]- \Re\int_{\R^{3}} g(\e x, |u_{n}|^{2}) u_{n} \bar{u} \eta_{\rho}\, dx,\\
&III_{n,\rho}:=-\Re\left[ M_{n}\iint_{\R^{6}} \frac{(u_{n}(x) - u_{n}(y)e^{\imath A_{\e}(\frac{x+y}{2})\cdot (x-y)})\overline{(u(x)-u(y)e^{\imath A_{\e}(\frac{x+y}{2})\cdot (x-y)})}}{|x-y|^{3+2s}} \eta_{\rho}(x) dxdy\right. \\
&\qquad \left.+ \int_{\R^{3}} V_{\e} u_{n}\bar{u} \eta_{\rho}\,dx\right]+ M_{n}\iint_{\R^{6}} \frac{|u(x) - u(y)e^{\imath A_{\e}(\frac{x+y}{2})\cdot (x-y)}|^{2}}{|x-y|^{3+2s}} \eta_{\rho}(x)\, dxdy + \int_{\R^{3}} V_{\e} |u|^{2}\eta_{\rho}\,dx\\
&\qquad =:-III^{1}_{n, \rho}+III^{2}_{n,\rho},\\
&IV_{n,\rho}:=\int_{\R^{3}} g(\e x, |u_{n}|^{2})|u_{n}|^{2} \eta_{\rho}\, dx- \Re\int_{\R^{3}} g(\e x, |u_{n}|^{2})u_{n} \bar{u} \eta_{\rho}\, dx.
\end{align*}
Let us prove that 
\begin{equation}\label{I0}
\lim_{\rho \rightarrow \infty} \limsup_{n\rightarrow \infty} |I_{n,\rho}|=0. 
\end{equation}
Firstly, we note that $I_{n,\rho}$ can be written as
\begin{align*}
I_{n,\rho}&=\langle J_{\e}'(u_{n}), u_{n}\eta_{\rho} \rangle \nonumber \\
&-\Re \left[M_{n}\iint_{\R^{6}} \frac{(u_{n}(x)- u_{n}(y)e^{\imath A_{\e}(\frac{x+y}{2})\cdot (x-y)})(\eta_{\rho}(x)- \eta_{\rho}(y))}{|x-y|^{3+2s}} \overline{u_{n}(y)}e^{-\imath A_{\e}(\frac{x+y}{2})\cdot (x-y)}\, dxdy\right].
\end{align*}
Since $(u_{n}\eta_{\rho})$ is bounded in $\h$, we have $\langle J_{\e}'(u_{n}), u_{n}\eta_{\rho} \rangle=o_{n}(1)$, and then
\begin{equation}\label{vin1}
I_{n, \rho}= o_{n}(1)- \Re \left[M_{n}\iint_{\R^{6}} \frac{(u_{n}(x)- u_{n}(y)e^{\imath A_{\e}(\frac{x+y}{2})\cdot (x-y)})(\eta_{\rho}(x)- \eta_{\rho}(y))}{|x-y|^{3+2s}} \overline{u_{n}(y)}e^{-\imath A_{\e}(\frac{x+y}{2})\cdot (x-y)}\, dxdy\right].
\end{equation}
Applying H\"older inequality and using the boundedness of $(u_{n})$ in $\h$ and \eqref{PS3} with $\eta_{R}= 1-\eta_{\rho}$, we can infer that
\begin{equation*}
\lim_{\rho \rightarrow \infty} \limsup_{n\rightarrow \infty} \left|M_{n} \iint_{\R^{6}} \frac{(u_{n}(x)- u_{n}(y)e^{\imath A_{\e}(\frac{x+y}{2})\cdot (x-y)})(\eta_{\rho}(x)- \eta_{\rho}(y))}{|x-y|^{3+2s}}\overline{u_{n}(y)}e^{-\imath A_{\e}(\frac{x+y}{2})\cdot (x-y)} dxdy\right| =0,
\end{equation*}
which together with \eqref{vin1} yields \eqref{I0}. 
Now, we note that 
\begin{align*}
II_{n, \rho}&=\langle J_{\e}'(u_{n}), u\eta_{\rho} \rangle \\
&- \Re \left[M_{n}\iint_{\R^{6}} \frac{(u_{n}(x)- u_{n}(y)e^{\imath A_{\e}(\frac{x+y}{2})\cdot (x-y)})(\eta_{\rho}(x)- \eta_{\rho}(y))}{|x-y|^{3+2s}} \overline{u(y)}e^{-\imath A_{\e}(\frac{x+y}{2})\cdot (x-y)}\, dxdy \right]. 
\end{align*}
Proceeding as in the previous case, we can show that
\begin{equation*}
\lim_{\rho \rightarrow \infty} \limsup_{n\rightarrow \infty} \left|M_{n}\iint_{\R^{6}} \frac{(u_{n}(x)- u_{n}(y)e^{\imath A_{\e}(\frac{x+y}{2})\cdot (x-y)})(\eta_{\rho}(x)- \eta_{\rho}(y))}{|x-y|^{3+2s}}\overline{u(y)}e^{-\imath A_{\e}(\frac{x+y}{2})\cdot (x-y)}\, dxdy\right| =0,
\end{equation*} 
and being $\langle J_{\e}'(u_{n}), u\eta_{\rho} \rangle=o_{n}(1)$, we obtain
\begin{equation}\label{II0}
\lim_{\rho \rightarrow \infty} \limsup_{n\rightarrow \infty} |II_{n, \rho}|=0. 
\end{equation}
Now, we show that
%On the other hand, from the weak convergence it follows that
\begin{equation}\label{III0}
\lim_{\rho \rightarrow \infty}\lim_{n\rightarrow \infty} |III_{n, \rho}|=0. 
\end{equation}
Firstly, we can use $M_{n}\rightarrow a+b\ell^{2}$ and the Dominated Convergence Theorem to see that
\begin{align}\label{SPLIT1}
\lim_{\rho\rightarrow \infty} \lim_{n\rightarrow \infty} III^{2}_{n, \rho}=(a+b\ell^{2})[u]_{A_{\e}}^{2}+\int_{\R^{3}} V_{\e}|u|^{2}dx=: L.
\end{align}
On the other hand, we can observe that $III_{n, \rho}^{1}$ can be written as follows:
\begin{align}\label{SPLIT}
III_{n, \rho}^{1} 
%&\Re\left[ M_{n}\iint_{\R^{6}} \frac{(u_{n}(x) - u_{n}(y)e^{\imath A_{\e}(\frac{x+y}{2})\cdot (x-y)})\overline{(u(x)-u(y)e^{\imath A_{\e}(\frac{x+y}{2})\cdot (x-y)})}}{|x-y|^{3+2s}} \eta_{\rho}(x) dxdy\right]+\Re \int_{\R^{3}} V_{\e} u_{n}\bar{u} \eta_{\rho}\,dx \nonumber\\
&=\Re\left[ M_{n}\iint_{\R^{6}} \frac{(u_{n}(x) - u_{n}(y)e^{\imath A_{\e}(\frac{x+y}{2})\cdot (x-y)})\overline{(u(x)\eta_{\rho}(x)-u(y) \eta_{\rho}(y) e^{\imath A_{\e}(\frac{x+y}{2})\cdot (x-y)})}}{|x-y|^{3+2s}} dxdy\right] \nonumber \\
&+\Re \int_{\R^{3}} V_{\e} u_{n}\bar{u} \eta_{\rho}\,dx \nonumber\\
&-\Re\left[ M_{n}\iint_{\R^{6}} \frac{(u_{n}(x) - u_{n}(y)e^{\imath A_{\e}(\frac{x+y}{2})\cdot (x-y)}) (\eta_{\rho}(x)-\eta_{\rho}(y))}{|x-y|^{3+2s}} \overline{u(y)} e^{-\imath A_{\e}(\frac{x+y}{2})\cdot (x-y)} dxdy\right] \nonumber \\
&=: A_{n, \rho}-B_{n, \rho}.
\end{align}
From the weak convergence  of $(u_{n})$  and $M_{n}\rightarrow a+b\ell^{2}$ we can obtain that
\begin{align*}
\lim_{n\rightarrow \infty}A_{n, \rho}&= \Re\left[ (a+b\ell^{2})\iint_{\R^{6}} \frac{(u(x) - u(y)e^{\imath A_{\e}(\frac{x+y}{2})\cdot (x-y)})\overline{(u(x)\eta_{\rho}(x)-u(y) \eta_{\rho}(y) e^{\imath A_{\e}(\frac{x+y}{2})\cdot (x-y)})}}{|x-y|^{3+2s}} dxdy\right] \\
&\quad +\Re \int_{\R^{3}} V_{\e} |u|^{2} \eta_{\rho}\,dx \\
&= \Re\left[ (a+b\ell^{2})\iint_{\R^{6}} \frac{|u(x) - u(y)e^{\imath A_{\e}(\frac{x+y}{2})\cdot (x-y)}|^{2}}{|x-y|^{3+2s}} \eta_{\rho}(x) dxdy\right] \\
&\quad+ \Re\left[ (a+b\ell^{2})\iint_{\R^{6}}  \frac{(u(x)- u(y)e^{\imath A_{\e}(\frac{x+y}{2})\cdot (x-y)})(\eta_{\rho}(x)- \eta_{\rho}(y))}{|x-y|^{3+2s}}\overline{u(y)}e^{-\imath A_{\e}(\frac{x+y}{2})\cdot (x-y)}\, dxdy \right]\\
&\quad+\Re \int_{\R^{3}} V_{\e} |u|^{2} \eta_{\rho}\,dx.
\end{align*}
Noting that
\begin{align}\label{PS3m}
&\left|\iint_{\R^{6}}  \frac{(u(x)- u(y)e^{\imath A_{\e}(\frac{x+y}{2})\cdot (x-y)})(\eta_{\rho}(x)- \eta_{\rho}(y))}{|x-y|^{3+2s}}\overline{u(y)}e^{-\imath A_{\e}(\frac{x+y}{2})\cdot (x-y)}\, dxdy\right|  \nonumber\\
&\leq [u]_{A_{\e}} \left(\iint_{\R^{6}} |u(y)|^{2}\frac{|\eta_{\rho}(x)-\eta_{\rho}(y)|^{2}}{|x-y|^{3+2s}} \, dxdy\right)^{1/2}\rightarrow 0 \mbox{ as } \rho\rightarrow \infty
\end{align}
(one can argue as in \eqref{PS3}), and using the Dominated Convergence Theorem we can deduce that
\begin{align*}
\lim_{\rho\rightarrow \infty}\lim_{n\rightarrow \infty}A_{n, \rho}=L.
%\Re\left[ (a+b\ell^{2})\iint_{\R^{6}} \frac{|u(x) - u(y)e^{\imath A_{\e}(\frac{x+y}{2})\cdot (x-y)}|^{2}}{|x-y|^{3+2s}} dxdy\right] +\Re \int_{\R^{3}} V_{\e} |u|^{2} \,dx=:L.
\end{align*}
Similarly to \eqref{PS3m}, we also have
\begin{align*}
\limsup_{n\rightarrow \infty}|B_{n, \rho}|\leq C \left(\iint_{\R^{6}} |u(y)|^{2}\frac{|\eta_{\rho}(x)-\eta_{\rho}(y)|^{2}}{|x-y|^{3+2s}} \, dxdy\right)^{1/2}\rightarrow 0 \mbox{ as } \rho\rightarrow \infty.
\end{align*}
From the above relations of limits we can infer that
\begin{align}\label{SPLIT2}
\lim_{\rho\rightarrow \infty} \lim_{n\rightarrow \infty} III^{1}_{n, \rho}=L.
\end{align}
Combining \eqref{SPLIT1} and \eqref{SPLIT2} and using the definition of $III_{n, \rho}$ we can conclude that \eqref{III0} holds true. 

In the light of $(g_{1})$ and $(g_{2})$ and the strong convergence of $|u_{n}|\rightarrow |u|$ in $L^{p}_{loc}(\R^{3}, \R)$ for $1\leq p<\frac{6}{3-2s}$ (by Theorem \ref{Sembedding}), we deduce that for any $\rho>R$ it holds
\begin{equation}\label{IV0}
\lim_{n\rightarrow \infty} |IV_{n, \rho}|=0. 
\end{equation}
Putting together \eqref{Phi}, \eqref{I0}, \eqref{II0}, \eqref{III0} and \eqref{IV0} we get
$$
0\leq \limsup_{n\rightarrow \infty} \int_{B_{R}} \Phi_{n}(x) dx\leq0,
$$
that is $\lim_{n\rightarrow \infty} \int_{B_{R}} \Phi_{n}(x) dx=0$ which yields \eqref{SPLITbr}.\\
{\bf Step 4} Conclusion. Using \eqref{T} we know that for each $\zeta>0$ there exists $R=R(\zeta)>\frac{C}{\zeta}$ such that 
\begin{equation}\label{ter8}
\limsup_{n\rightarrow \infty} \left[ \int_{\R^{3}\setminus B_{R}} \,dx \int_{\R^{3}} a\frac{|u_{n}(x)- u_{n}(y)e^{\imath A_{\e}(\frac{x+y}{2})\cdot (x-y)}|^{2}}{|x-y|^{3+2s}}\,dy + \int_{\R^{3}\setminus B_{R}} V(\e x) |u_{n}|^{2} \,dx\right] <\zeta. 
\end{equation}
Taking into account $u_{n}\rightharpoonup u$ in $\h$, \eqref{ter8} and \eqref{SPLITbr} we can infer
\begin{align*}
\|u\|_{\e}^{2} &\leq \liminf_{n\rightarrow \infty} \|u_{n}\|_{\e}^{2} \leq  \limsup_{n\rightarrow \infty} \|u_{n}\|_{\e}^{2}  \\
%&\leq \limsup_{n\rightarrow \infty} \left[ \iint_{\R^{6}} \frac{|u_{n}(x)- u_{n}(y)|^{2}}{|x-y|^{3+2s}}\, dxdy + \int_{\R^{3}} V(\e x) \,u_{n}^{2} dx \right]\\
&= \limsup_{n\rightarrow \infty} \Bigl[ \, \int_{B_{R}} dx \int_{\R^{3}} a\frac{|u_{n}(x)- u_{n}(y)e^{\imath A_{\e}(\frac{x+y}{2})\cdot (x-y)}|^{2}}{|x-y|^{3+2s}}\, dy + \int_{B_{R}} V_{\e} \,|u_{n}|^{2} dx \\
&\quad\quad+  \int_{\R^{3}\setminus B_{R}} dx \int_{\R^{3}} a\frac{|u_{n}(x)- u_{n}(y)e^{\imath A_{\e}(\frac{x+y}{2})\cdot (x-y)}|^{2}}{|x-y|^{3+2s}}\, dy + \int_{\R^{3}\setminus B_{R}} V_{\e} |u_{n}|^{2} dx \, \Bigr]\\
&\leq \int_{B_{R}} dx \int_{\R^{3}} a\frac{|u(x)-u(y)e^{\imath A_{\e}(\frac{x+y}{2})\cdot (x-y)}|^{2}}{|x-y|^{3+2s}} \, dy + \int_{B_{R}} V_{\e} |u|^{2} dx +\zeta.
%\leq \|u\|_{\e}^{2}+ \zeta,
\end{align*}
Since $R\rightarrow \infty$ as $\zeta\rightarrow 0$, we get
\begin{equation*}
\|u\|_{\e}^{2}\leq \liminf_{n\rightarrow \infty} \|u_{n}\|_{\e}^{2} \leq \limsup_{n\rightarrow \infty} \|u_{n}\|_{\e}^{2}\leq \|u\|_{\e}^{2}, 
\end{equation*}
which implies $\|u_{n}\|_{\e}\rightarrow \|u\|_{\e}$. Recalling that $\h$ is a Hilbert space, we can deduce that $u_{n}\rightarrow u$ in $\h$ as $n\rightarrow \infty$. 
\end{proof}

Since we are looking for multiple critical points of the functional $J_{\e}$, we shall consider it constrained to an appropriated subset of $\h$.
More precisely, we define the Nehari manifold associated to (\ref{MPe}), that is
\begin{equation*}
\mathcal{N}_{\e}:= \{u\in \h \setminus \{0\} : \langle J_{\e}'(u), u \rangle =0\},
\end{equation*}
and we indicate by $\mathcal{N}_{\mu}$ the Nehari manifold associated to \eqref{AP0}.
Moreover, it is easy to show (see \cite{W}) that $c_{\e}$ can be characterized as follows:
$$
c_{\e}=\inf_{u\in \h\setminus\{0\}} \sup_{t\geq 0} J_{\e}(t u)=\inf_{u\in \N_{\e}} J_{\e}(u).
$$
In what follows, we denote by $c_{\mu}$ the minimax level for the autonomous problem \eqref{AP0}.

From the growth conditions of $g$, we can see that 
for a fixed $u\in \mathcal{N}_{\e}$
\begin{align*}
0&\geq\|u\|_{\e}^{2}-\int_{\R^{3}} g(\e_{n} x, |u|^{2})|u|^{2}\,dx\\
&\geq \|u\|_{\e}^{2}-\frac{1}{k} \int_{\R^{3}} V_{\e}(x)|u|^{2}\, dx-C\|u\|_{\e}^{q} \\
&\geq \min\left\{a, \frac{k-1}{k}\right\}\|u\|^{2}_{\e}-C\|u\|_{\e}^{q},
\end{align*}
so there exists $r>0$ independent of  $u$ such that 
\begin{equation}\label{uNr}
\|u\|_{\e}\geq r \mbox{ for all } u\in \mathcal{N}_{\e}.
\end{equation}
Now, we prove the following result.
\begin{prop}\label{propPSc}
Let $c\in \R$. Then, the functional $J_{\e}$ restricted to $\mathcal{N}_{\e}$ satisfies the $(PS)_{c}$ condition at the level $c$.
\end{prop}
\begin{proof}
Let $(u_{n})\subset \mathcal{N}_{\e}$ be such that $J_{\e}(u_{n})\rightarrow c$ and $\|J'_{\e}(u_{n})_{|\mathcal{N}_{\e}}\|_{*}=o_{n}(1)$. Then (see \cite{W}) we can find $(\lambda_{n})\subset \R$ such that
\begin{equation}\label{AFT}
J'_{\e}(u_{n})=\lambda_{n} T'_{\e}(u_{n})+o_{n}(1),
\end{equation}
where $T_{\e}: \h\rightarrow \R$ is defined as
\begin{align*}
T_{\e}(u)=\|u\|_{\e}^{2}+b[u]^{4}_{A_{\e}}-\int_{\R^{3}} g(\e x, |u|^{2})|u|^{2}\, dx.
\end{align*}
In view of $\langle J'_{\e}(u_{n}), u_{n}\rangle=0$, $g(\e x, |u|^{2})$ is constant in $\Lambda_{\e}^{c}\cap \{|u|^{2}>T_{a'}\}$,  and using the definitions of $g$, the monotonicity of $\eta$ and $(f_4)$, we obtain
\begin{align}
&\langle T'_{\e}(u_{n}), u_{n}\rangle \nonumber\\
&=2\|u_{n}\|_{\e}^{2}+4 b[u_{n}]^{4}_{A_{\e}}-2\int_{\R^{3}} g'(\e x, |u_{n}|^{2})|u_{n}|^{4}\, dx-2\int_{\R^{3}} g(\e x, |u_{n}|^{2})|u_{n}|^{2}\, dx \nonumber\\
&=-2\|u_{n}\|^{2}_{\e}+2\int_{\R^{3}} g(\e x, |u_{n}|^{2})|u_{n}|^{2}\, dx-2\int_{\R^{3}} g'(\e x, |u_{n}|^{2})|u_{n}|^{4}\, dx \nonumber \\
%&= -2 \|u_{n}\|^{2}_{\e} + 2 \int_{\R^{3}\setminus \Lambda_{\e}} g(\e x, |u_{n}|^{2}) |u_{n}|^{2} \, dx +  2 \int_{\Lambda_{\e}} g(\e x, |u_{n}|^{2}) |u_{n}|^{2} \, dx - 2\int_{\R^{3}} g'(\e x, |u_{n}|^{2})|u_{n}|^{4}\, dx \nonumber \\
%&\leq -2 \|u_{n}\|^{2}_{\e}+ \frac{2}{k} \int_{\R^{3}} V_{\e}(x) |u_{n}|^{2} \, dx + 2 \int_{\Lambda_{\e}} \left[g(\e x, |u_{n}|^{2}) |u_{n}|^{2}- g'(\e x, |u_{n}|^{2}) |u_{n}|^{4}\right] \, dx \nonumber \\
&\leq -C\int_{\Lambda_{\e}\cup \{|u_{n}|^{2}<t_{a'}\}}  |u_{n}|^{\sigma} dx \nonumber \\
&\leq -C\int_{\Lambda_{\e}} |u_{n}|^{\sigma} dx \label{22ZS1}.
%&\leq \left( \frac{2}{k}- 2\right) \|u_{n}\|^{2}_{\e} + 2 \int_{\Lambda_{\e}} \left[g(\e x, |u_{n}|^{2}) |u_{n}|^{2}- g'(\e x, |u_{n}|^{2}) |u_{n}|^{4}\right] \, dx\leq 0 \label{22ZS1}.
\end{align}
 %On the other hand, using the definition of $g$ we know that
%\begin{equation*}
%g'(x,t)=0 \quad \forall t\geq a' \quad \mbox{ and } \quad g'(x,t)= \tilde{f}'(t) \quad \forall x\in \R^{3}\setminus \Lambda. 
%\end{equation*}
%Therefore, by $(g_2)$, $f', \tilde{f}' \in C(\R^{3})$, we obtain
%\begin{align*}
%\langle T'_{\e}(u_{n}), u_{n}\rangle&\geq -2\int_{\R^{3}} g'(\e x, |u_{n}|^{2})|u_{n}|^{4}\, dx-2\int_{\R^{3}} g(\e x, |u_{n}|^{2})|u_{n}|^{2}\, dx \\
%&\geq -2\int_{\R^{3} \cap \{|u_{n}|\leq a'\}} g'(\e x, |u_{n}|^{2})|u_{n}|^{4}\, dx - 2 \int_{\R^{3}} f(|u_{n}|^{2}) |u_{n}|^{2}\, dx \\
%&\geq -2 \int_{\Lambda \cap \{|u_{n}|\leq a'\}} f'(|u_{n}|^{2})|u_{n}|^{4} \, dx - 2\int_{\Lambda^{c}\cap \{|u_{n}|\leq a'\} } \tilde{f}'(|u_{n}|^{2})|u_{n}|^{4} \, dx \\
%&\quad - 2 \int_{\R^{3}} f(|u_{n}|^{2}) |u_{n}|^{2}\, dx\\
%&\geq - C_{a}\int_{\R^{3}} |u_{n}|^{\sigma}\, dx - 2\int_{\R^{3}} f(|u_{n}|^{2}) |u_{n}|^{2}\, dx. 
%\end{align*}
%Taking into account $(f_1)$-$(f_2)$, Theorem \ref{Sembedding} and the boundedness of $(u_{n})$ in $\h$, 
Since $(u_{n})$ is bounded in $\h$, we may assume that $\langle T'_{\e}(u_{n}), u_{n}\rangle\rightarrow \ell\leq 0$. If $\ell=0$, from \eqref{22ZS1} it follows that $u_{n}\rightarrow 0$ in $L^{\sigma}(\Lambda_{\e}, \R)$. Using $\langle J'_{\e}(u_{n}), u_{n}\rangle=0$, $g$ is subcritical and $(g_3)$-(ii) we have
\begin{equation*}
\|u_{n}\|^{2}_{\e} \leq  \int_{\Lambda_{\e}^{c}} g(\e x, |u_{n}|^{2}) |u_{n}|^{2} \,dx+o_{n}(1)\leq \frac{1}{K}\int_{\R^{3}} V_{\e}(x)|u_{n}|^{2}\, dx+ o_{n}(1)
\end{equation*}
that is $\|u_{n}\|_{\e}\rightarrow 0$ which contradicts \eqref{uNr}. Consequently, $\ell<0$ and in the light of \eqref{AFT} we can deduce that $\lambda_{n}\rightarrow 0$. Hence, $u_{n}$ is a $(PS)_{c}$ sequence for the unconstrained functional and we can apply Lemma \ref{PSc} to get the thesis.
\end{proof}

As a byproduct of the above proof we have the following result:
\begin{cor}\label{cor}
The critical points of the functional $J_{\e}$ on $\mathcal{N}_{\e}$ are critical points of $J_{\e}$.
\end{cor}

%\noindent
%Taking into account Lemma \ref{MPG}, we can define the mountain pass level
%$$
%c_{\e}=\inf_{\gamma\in \Gamma_{\e}} \max_{t\in [0, 1]} J_{\e}(\gamma(t))
%$$
%where
%$$
%\Gamma_{\e}=\{\gamma\in C([0, 1], \h): \gamma(0)=0 \mbox{ and } J_{\e}(\gamma(1))<0\}.
%$$
%By applying Mountain Pass Theorem \cite{AR}, we can see that there exists $u_{\e}\in \h\setminus\{0\}$ such that $J_{\e}(u_{\e})=c_{\e}$ and $J'_{\e}(u_{\e})=0$. In similar fashion, one can prove that also $J_{0}$ has a mountain pass geometry, so we denote by $c_{V_{0}}$ the mountain pass level associated to $J_{0}$.\\
At this point, we provide some useful results about Kirchhoff autonomous problems \eqref{AP0}.
We begin proving the following Lions compactness result.
\begin{lem}\label{LionsFS}
Let $(u_{n})\subset H^{s}_{\mu}$ be a $(PS)_{c}$ sequence for $J_{\mu}$. Then one of the following conclusions holds:
\begin{compactenum}[$(i)$]
\item $u_{n}\rightarrow 0$ in $H^{s}_{\mu}$;
\item there exists a sequence $(y_{n})\subset \R^{3}$ and constants $R, \beta>0$ such that
$$
\liminf_{n\rightarrow \infty} \int_{B_{R}(y_{n})} |u_{n}|^{2}dx\geq \beta>0.
$$
\end{compactenum}
\end{lem}
\begin{proof}
Assume that $(ii)$ does not occur. Arguing as in the proof of Lemma \ref{PSc} we can see that $(u_{n})$ is bounded in $H^{s}_{\mu}$. Then we can use Lemma \ref{Lions} to deduce that $u_{n}\rightarrow 0$ in $L^{r}(\R^{3}, \R)$ for all $r\in (2, \2)$. In view of $(f_1)$-$(f_2)$ we get $\int_{\R^{3}} f(u^{2}_{n})u^{2}_{n}dx=o_{n}(1)$. This fact combined with $\langle J'_{\mu}(u_{n}), u_{n}\rangle=o_{n}(1)$ yields $\|u_{n}\|^{2}_{\mu}\leq \|u_{n}\|^{2}_{\mu}+b[u_{n}]^{4}=\int_{\R^{3}} f(u_{n}^{2})u_{n}^{2}dx+o_{n}(1)=o_{n}(1)$.
\end{proof}

\noindent 
Therefore, we can prove an existence result for the autonomous Kirchhoff problem.
\begin{lem}\label{FS}
Fo all $\mu>0$, there exists a positive ground state solution of \eqref{AP0}.
%Let $(u_{n})\subset \mathcal{N}_{\mu}$ be a sequence satisfying $J_{\mu}(u_{n})\rightarrow c_{\mu}\in \R$. Then, up to subsequences, the following alternatives holds:
%\begin{compactenum}[(i)]
%\item $(u_{n})$ strongly converges in $H^{s}_{\mu}$, 
%\item there exists a sequence $(\tilde{y}_{n})\subset \R^{3}$ such that,  up to a subsequence, $v_{n}(x)=u_{n}(x+\tilde{y}_{n})$ converges strongly in $H^{s}_{\mu}$.
%\end{compactenum}
%In particular, there exists a minimizer $w\in H^{s}_{\mu}$ for $J_{\mu}$ with $J_{\mu}(w)=c_{\mu}$.
\end{lem}
\begin{proof}
It is easy to check that $J_{\mu}$ has a mountain pass geometry, so there exists a sequence $(u_{n})\subset H^{s}_{\mu}$ such that $J_{\mu}(u_{n})\rightarrow c_{\mu}$ and $J'_{\mu}(u_{n})\rightarrow 0$.  Thus, $(u_{n})$ is bounded in $H^{s}_{\mu}$ and we may assume that $u_{n}\rightharpoonup u$ in $H^{s}_{\mu}$ and $[u_{n}]^{2}\rightarrow B^{2}$. 
Suppose that $u\neq 0$. Since $\langle J'_{\mu}(u_{n}), \varphi\rangle=o_{n}(1)$ we can deduce that for all $\varphi\in C^{\infty}_{c}(\R^{3}, \R)$
\begin{equation}\label{HZman}
 \int_{\R^{3}} a(-\Delta)^{\frac{s}{2}} u (-\Delta)^{\frac{s}{2}}\varphi+\mu u \varphi \,dx+bB^{2} \left( \int_{\R^{3}} (-\Delta)^{\frac{s}{2}} u (-\Delta)^{\frac{s}{2}}\varphi \,dx\right)-\int_{\R^{3}} f(u^{2})u\varphi \,dx=0.
\end{equation}
Let us note that $B^{2}\geq [u]^{2}$ by Fatou's Lemma. If by contradiction $B^{2}>[u]^{2}$, we may use \eqref{HZman} to deduce that $\langle J'_{\mu}(u), u\rangle<0$. Moreover, conditions $(f_1)$-$(f_2)$ imply that $\langle J'_{\mu}(t u), t u\rangle>0$ for small $t>0$. Then there exists $t_{0}\in (0, 1)$ such that $t_{0} u\in \N_{\mu}$ and $\langle J'_{\mu}(t_{0} u), t_{0} u\rangle=0$. Using Fatou's Lemma, $t_{0}\in (0, 1)$ and  $\frac{1}{4}f(t)t-\frac{1}{2}F(t)$ is increasing for $t>0$ (by $(f_3)$ and $(f_4)$) we get
\begin{align*}
c_{\mu}&\leq J_{\mu}(t_{0} u)-\frac{1}{4} \langle J'_{\mu}(t_{0} u), t_{0} u\rangle<\liminf_{n\rightarrow \infty} \left[J_{\mu}(u_{n})-\frac{1}{4} \langle J'_{\mu}(u_{n}), u_{n}\rangle\right]=c_{\mu}
\end{align*}
which gives a contradiction. Therefore $B^{2}= [u]^{2}$ and we deduce that $J'_{\mu}(u)=0$. Hence  $u\in \N_{\mu}$. Using the fact that $\langle J'_{\mu}(u), u^{-}\rangle=0$ and $(f_1)$ we can see that $u\geq 0$ in $\R^{3}$. Moreover we can argue as in Lemma \ref{moser} to infer that $u\in L^{\infty}(\R^{3}, \R)$. Since $u$ satisfies
$$
(-\Delta)^{s}u=(a+b[u]^{2})^{-1}[f(u^{2})u-\mu u]\in L^{\infty}(\R^{3}, \R),
$$
and $s>\frac{3}{4}>\frac{1}{2}$, we obtain $u\in C^{1, \gamma}(\R^{3}, \R)\cap L^{\infty}(\R^{3},\R)$, for some $\gamma>0$ (see \cite{S}) and that $u>0$ by the maximum principle. Now we prove that $J_{\mu}(u)=c_{\mu}$. Indeed, using $u\in \N_{\mu}$, $(f_3)$ and Fatou's Lemma we have
\begin{align*}
c_{\mu}&\leq J_{\mu}(u)-\frac{1}{4}\langle J'_{\mu}(u), u\rangle\\
&\leq \liminf_{n\rightarrow \infty} \left[\frac{1}{4}\|u_{n}\|_{\mu}^{2}+\int_{\R^{3}} \frac{1}{4} f(u^{2}_{n})u^{2}_{n}-\frac{1}{2} F(u^{2}_{n}) dx\right] \\
&=\liminf_{n\rightarrow \infty} J_{\mu}(u_{n})-\frac{1}{4}\langle J'_{\mu}(u_{n}), u_{n}\rangle \\
&=c_{\mu}.
\end{align*}
Now, we consider the case $u=0$. Since $c_{\mu}>0$ and $J_{\mu}$ is continuous, we can see that $\|u_{n}\|_{\mu}\nrightarrow 0$. From Lemma \ref{LionsFS} it follows that we can  define $v_{n}(x)=u_{n}(x+y_{n})$ such that $v_{n}\rightharpoonup v$ in $H^{s}_{\mu}$ for some $v\neq 0$. Then we can argue as in the previous case to get the thesis.
\end{proof}

\noindent
The next result shows an interesting relation between $c_{\e}$ and $c_{V_{0}}$.
\begin{lem}\label{AMlem1}
The numbers $c_{\e}$ and $c_{V_{0}}$ satisfy the following inequality
$$
\limsup_{\e\rightarrow 0} c_{\e}\leq c_{V_{0}}.
$$
\end{lem}
\begin{proof}
In the light of Lemma \ref{FS}, we can find a positive ground state $w\in H^{s}_{V_{0}}$ to \eqref{AP0}, that is $J'_{V_{0}}(w)=0$ and $J_{V_{0}}(w)=c_{V_{0}}$. 
Since $w\in C^{1, \gamma}(\R^{3}, \R)\cap L^{\infty}(\R^{3},\R)$, for some $\gamma>0$, we get $|w(x)|\rightarrow 0$ as $|x|\rightarrow \infty$. Observing that $w$ satisfies
$$
(-\Delta)^{s}w+\frac{V_{0}}{a+bM^{2}}w=(a+b[w]^{2})^{-1}[f(w^{2})w-V_{0} w]+\frac{V_{0}}{a+bM^{2}}w \mbox{ in } \R^{3},
$$
where $0<a\leq a+b[u]^{2}\leq a+bM^{2}$, we can
%and we can find $R>0$ such that $(-\Delta)^{s}w+\frac{V_{0}}{2(a+bM)}w\leq 0$ in $|x|>R$, where $M=[u]^{2}$. By using Lemma 4.3 in \cite{FQT} we know that there exists a positive continuous function $\tilde{w}$ such that for $|x|>R$ (taking $R$ larger if it is necessary), it holds $(-\Delta)^{s}\tilde{w}+\frac{V_{0}}{2}\tilde{w}=0$ and $\tilde{w}(x)=\frac{C_{0}}{|x|^{3+2s}}$. 
%In view of the continuity of $w$ and $\tilde{w}$ there exists some constant $C_{1}>0$ such that $z=w-C_{1}\tilde{w}\leq 0$ on $|x|=R$.
%Moreover, we can see that $(-\Delta)^{s}z+\frac{V_{0}}{2}z\geq 0$ in $|x|\geq R$. By using the maximum principle we can deduce that $z\leq 0$ in $|x|\geq R$, that is 
argue as in Lemma $4.3$ in \cite{FQT} to deduce the following decay estimate
\begin{equation}\label{remdecay}
0<w(x)\leq \frac{C}{|x|^{3+2s}} \quad \mbox{ for } |x|>>1.
\end{equation}
Now, let $\eta\in C^{\infty}_{c}(\R^{3}, [0,1])$ be a cut-off function such that $\eta=1$ in a neighborhood of zero $B_{\frac{\delta}{2}}$ and $\supp(\eta)\subset B_{\delta}\subset \Lambda$ for some $\delta>0$. 
Let us define $w_{\e}(x):=\eta_{\e}(x)w(x) e^{\imath A(0)\cdot x}$, with $\eta_{\e}(x)=\eta(\e x)$ for $\e>0$, and we note that $|w_{\e}|=\eta_{\e}w$ and $w_{\e}\in \h$ in view of Lemma \ref{aux}. 
Let us verify that
\begin{equation}\label{limwr}
\lim_{\e\rightarrow 0}\|w_{\e}\|^{2}_{\e}=\|w\|_{V_{0}}^{2}\in(0, \infty).
\end{equation}
From the Dominated Convergence Theorem it follows that $\int_{\R^{3}} V_{\e}(x)|w_{\e}|^{2}dx\rightarrow \int_{\R^{3}} V_{0} |w|^{2}dx$. Thus, it is only need to prove that
\begin{equation}\label{limwr*}
\lim_{\e\rightarrow 0}[w_{\e}]^{2}_{A_{\e}}=[w]^{2}.
\end{equation}
By Lemma $5$ in \cite{PP}, we know that 
\begin{equation}\label{PPlem}
[\eta_{\e} w]\rightarrow [w] \mbox{ as } \e\rightarrow 0.
\end{equation}
On the other hand
\begin{align*}
[w_{\e}]_{A_{\e}}^{2}
&=\iint_{\R^{6}} \frac{|e^{\imath A(0)\cdot x}\eta_{\e}(x)w(x)-e^{\imath A_{\e}(\frac{x+y}{2})\cdot (x-y)}e^{\imath A(0)\cdot y} \eta_{\e}(y)w(y)|^{2}}{|x-y|^{3+2s}} dx dy \nonumber \\
&=[\eta_{\e} w]^{2}
+\iint_{\R^{6}} \frac{\eta_{\e}^2(y)w^2(y) |e^{\imath [A_{\e}(\frac{x+y}{2})-A(0)]\cdot (x-y)}-1|^{2}}{|x-y|^{3+2s}} dx dy\\
&\quad+2\Re \iint_{\R^{6}} \frac{(\eta_{\e}(x)w(x)-\eta_{\e}(y)w(y))\eta_{\e}(y)w(y)(1-e^{-\imath [A_{\e}(\frac{x+y}{2})-A(0)]\cdot (x-y)})}{|x-y|^{3+2s}} dx dy \\
&=: [\eta_{\e} w]^{2}+X_{\e}+2Y_{\e}.
\end{align*}
In the light of 
$|Y_{\e}|\leq [\eta_{\e} w] \sqrt{X_{\e}}$ and \eqref{PPlem}, it is enough to see that $X_{\e}\rightarrow 0$ as $\e\rightarrow 0$ to deduce that \eqref{limwr*} holds.
For all $0<\beta<\alpha/({1+\alpha-s})$, we get
\begin{equation}\label{Ye}
\begin{split}
X_{\e}
&\leq \int_{\R^{3}} w^{2}(y) dy \int_{|x-y|\geq\e^{-\beta}} \frac{|e^{\imath [A_{\e}(\frac{x+y}{2})-A(0)]\cdot (x-y)}-1|^{2}}{|x-y|^{3+2s}} dx\\
&\quad+\int_{\R^{3}} w^{2}(y) dy  \int_{|x-y|<\e^{-\beta}} \frac{|e^{\imath [A_{\e}(\frac{x+y}{2})-A(0)]\cdot (x-y)}-1|^{2}}{|x-y|^{3+2s}} dx \\
&=:X^{1}_{\e}+X^{2}_{\e}.
\end{split}
\end{equation}
Since $|e^{\imath t}-1|^{2}\leq 4$ and $w\in H^{s}(\R^{3}, \R)$, we have
\begin{equation}\label{Ye1}
X_{\e}^{1}\leq C \int_{\R^{3}} w^{2}(y) dy \int_{\e^{-\beta}}^\infty \rho^{-1-2s} d\rho\leq C \e^{2\beta s} \rightarrow 0.
\end{equation}
Observing that $|e^{\imath t}-1|^{2}\leq t^{2}$ for all $t\in \R$, $A\in C^{0,\alpha}(\R^3,\R^3)$ for $\alpha\in(0,1]$, and $|x+y|^{2}\leq 2(|x-y|^{2}+4|y|^{2})$, we can deduce that
\begin{equation}\label{Ye2}
\begin{split}
X^{2}_{\e}&
	\leq \int_{\R^{3}} w^{2}(y) dy  \int_{|x-y|<\e^{-\beta}} \frac{|A_{\e}\left(\frac{x+y}{2}\right)-A(0)|^{2} }{|x-y|^{3+2s-2}} dx \\
	&\leq C\e^{2\alpha} \int_{\R^{3}} w^{2}(y) dy  \int_{|x-y|<\e^{-\beta}} \frac{|x+y|^{2\alpha} }{|x-y|^{3+2s-2}} dx \\
	&\leq C\e^{2\alpha} \left(\int_{\R^{3}} w^{2}(y) dy  \int_{|x-y|<\e^{-\beta}} \frac{1 }{|x-y|^{3+2s-2-2\alpha}} dx\right.\\
	&\quad+ \left. \int_{\R^{3}} |y|^{2\alpha} w^{2}(y) dy  \int_{|x-y|<\e^{-\beta}} \frac{1}{|x-y|^{3+2s-2}} dx\right) \\
	&=: C\e^{2\alpha} (X^{2, 1}_{\e}+X^{2, 2}_{\e}).
	\end{split}
	\end{equation}	
	Hence
	\begin{equation}\label{Ye21}
	X^{2, 1}_{\e}
	= C  \int_{\R^{3}} w^{2}(y) dy \int_0^{\e^{-\beta}} \rho^{1+2\alpha-2s} d\rho
	\leq C\e^{-2\beta(1+\alpha-s)}.
	\end{equation}
	On the other hand, using \eqref{remdecay}, we have
	\begin{equation}\label{Ye22}
	\begin{split}
	 X^{2, 2}_{\e}
	 &\leq C  \int_{\R^{3}} |y|^{2\alpha} w^{2}(y) dy \int_0^{\e^{-\beta}}\rho^{1-2s} d\rho  \\
	&\leq C \e^{-2\beta(1-s)} \left[\int_{B_1(0)}  w^{2}(y) dy + \int_{B_1^c(0)} \frac{1}{|y|^{2(3+2s)-2\alpha}} dy \right]  \\
	&\leq C \e^{-2\beta(1-s)}.
	\end{split}
	\end{equation}
	From \eqref{Ye}, \eqref{Ye1}, \eqref{Ye2}, \eqref{Ye21} and \eqref{Ye22} it follows that that $X_{\e}\rightarrow 0$,  that is \eqref{limwr} holds true.  
%Moreover, by \eqref{PPlem}, the Dominated Convergence Theorem, $\h$ is a Hilbert space, we can see that $|w_{\e}|=\eta_{\e}w$ strongly converges to $w$ in $H^{s}(\R^{3}, \R)$, so we deduce that
%\begin{equation}\label{PoisC}
%\lim_{\e\rightarrow 0} \int_{\R^{3}}|w_{\e}|^{\2} dx=\int_{\R^{3}} |w|^{\2}dx
%\end{equation}
%and
%\begin{equation}\label{Poisslim}
%\lim_{\e\rightarrow 0} \int_{\R^{3}}\phi_{|w_{\e}|}^{t} |w_{\e}|^{2}dx=\int_{\R^{3}}\phi_{w}^{t} w^{2}dx.
%\end{equation}
Now, let $t_{\e}>0$ be the unique number such that 
\begin{equation*}
J_{\e}(t_{\e} w_{\e})=\max_{t\geq 0} J_{\e}(t w_{\e}).
\end{equation*}
Clearly $t_{\e}$ satisfies 
\begin{equation}\label{AS1}
t_{\e}^{2}\|w_{\e}\|_{\e}^{2}+t_{\e}^{4}[w_{\e}]^{4}_{A_{\e}}=\int_{\R^{3}} g(\e x, t_{\e}^{2} |w_{\e}|^{2}) |t_{\e}w_{\e}|^{2}dx=\int_{\R^{3}} f(t_{\e}^{2} |w_{\e}|^{2}) |t_{\e}w_{\e}|^{2} dx,
\end{equation}
where we used $supp(\eta)\subset \Lambda$ and $g(t)=f(t)$ on $\Lambda$.\\
Now, we show that $t_{\e}\rightarrow 1$ as $\e\rightarrow 0$. Since $\eta=1$ in $B_{\frac{\delta}{2}}$, $w$ is a continuous positive function and $\frac{f(t)}{t}$ is increasing for $t>0$ by $(f_4)$, we can deduce
\begin{align*}
\frac{1}{t_{\e}^{2}}\|w_{\e}\|_{\e}^{2}+b [w_{\e}]_{A_{\e}}^{4} &= \int_{\R^{3}} \frac{f(t_{\e}^{2}|w_{\e}|^{2})}{t_{\e}^{2}|w_{\e}|^{2}} |w_{\e}|^{4} dx\\ 
&\geq \frac{f(t_{\e}^{2}\alpha^{2}_{0})}{t_{\e}^{2}\alpha^{2}_{0}}\int_{B_{\frac{\delta}{2}}}|w|^{4}dx
\end{align*}
where $\alpha_{0}:=\min_{\bar{B}_{\frac{\delta}{2}}} w>0$. Therefore, if $t_{\e}\rightarrow \infty$ as $\e\rightarrow 0$, we can use \eqref{limwr} to see that $b[w]^{2}=\infty$, an absurd.
On the other hand, if $t_{\e}\rightarrow 0$ as $\e\rightarrow 0$, by \eqref{AS1}, the growth assumptions on $g$ and  \eqref{limwr} yield $\|w\|_{V_{0}}^{2}= 0$, which is impossible.
Thus, $t_{\e}\rightarrow t_{0}\in (0, \infty)$ as $\e\rightarrow 0$.

Letting the limit as $\e\rightarrow 0$ in \eqref{AS1} and by \eqref{limwr}, we can see that 
\begin{equation*}
%\label{AS2}
\frac{1}{t_{0}^{2}}\|w\|_{V_{0}}^{2}+b[w]^{4}=\int_{\R^{3}} \frac{f(t_{0}^{2} w^{2})}{(t_{0}^{2}w^{2})} w^{4} dx.
\end{equation*}
By $w\in \mathcal{N}_{V_{0}}$ and 
 $(f_4)$, we can conclude that $t_{0}=1$. Applying the Dominated Convergence Theorem, we can see that $\lim_{\e\rightarrow 0} J_{\e}(t_{\e} w_{\e})=J_{V_{0}}(w)=c_{V_{0}}$.
Using $c_{\e}\leq \max_{t\geq 0} J_{\e}(t w_{\e})=J_{\e}(t_{\e} w_{\e})$, we can infer  that
$\limsup_{\e\rightarrow 0} c_{\e}\leq c_{V_{0}}$.
\end{proof}
%Taking into account Lemma \ref{MPG}, Lemma \ref{PSc}, Theorem \ref{AMlem1} and Mountain Pass Theorem \cite{AR}, we can deduce the existence of a nontrivial solution to \eqref{MPe} for $\e>0$ small. More precisely
%\begin{thm}
%Assume that $(V_1)$-$(V_2)$ and $(f_1)$-$(f_4)$ hold. Then there exists $\e_{0}>0$ such that, for any $\e\in (0, \e_{0})$, problem \eqref{Pe} admits a nontrivial solution.
%\end{thm} 

%\begin{lem}\label{AMlem2}
%There exist $r, \beta, \e^{*}>0$ and $(\tilde{y}_{\e})\subset \R^{N}$ such that
%$$
%\int_{B_{r}(\tilde{y}_{\e})}|u_{\e}|^{2}dx\geq r \mbox{ for all } \e\in (0, \e^{*}).
%$$
%\end{lem}
%\begin{proof}
%Since $J'_{\e}(u_{\e})=0$, we can use $(g_1)$ and $(g_2)$ to see that there exists $\alpha_{0}>0$ independent of $\e$ such that
%\begin{equation}\label{2.12AM}
%\|u_{\e}\|^{2}_{\e}\geq \alpha_{0} \quad \forall \e>0.
%\end{equation}
%To show the lemma, it is enough to verify that for any sequence $(\e_{n})\subset (0, \infty)$ with $\e_{n}\rightarrow 0$ it holds
%$$
%\lim_{n\rightarrow \infty} \sup_{y\in \R^{N}}\int_{B_{r}(y)} |u_{\e_{n}}|^{2} dx=0
%$$
%for any $r>0$. If by contradiction there exists $r>0$ for which the above limit is true, then we can use a Lion's type result ( see Lemma $2.2$ in \cite{FQT}) to deduce that $|u_{\e_{n}}|\rightarrow 0$ in $L^{q}(\R^{N}, \R)$ for any $q\in (2, 2^{*}_{s})$. Then, using $(g_1)$-$(g_2)$, it is easy to check that $\|u_{\e_{n}}\|_{\e_{n}}^{2}\rightarrow 0$ as $n\rightarrow \infty$. This is impossible in view of \eqref{2.12AM}.
%\end{proof}

\section{Multiplicity result for the modified problem}
In this section we make use of the Ljusternik-Schnirelmann category theory to obtain multiple solutions to \eqref{MPe}. 
In particular, we relate the number of positive solutions of \eqref{MPe} to the topology of the set $M$.
For this reason, we take $\delta>0$ such that
$$
M_{\delta}=\{x\in \R^{3}: {\rm dist}(x, M)\leq \delta\}\subset \Lambda,
$$
and we consider $\eta\in C^{\infty}_{0}(\R_{+}, [0, 1])$ such that $\eta(t)=1$ if $0\leq t\leq \frac{\delta}{2}$ and $\eta(t)=0$ if $t\geq \delta$.\\
For any $y\in \Lambda$, we introduce (see \cite{AD})
$$
\Psi_{\e, y}(x)=\eta(|\e x-y|) w\left(\frac{\e x-y}{\e}\right)e^{\imath \tau_{y} \left( \frac{\e x-y}{\e} \right)},
$$
where $\tau_{y}(x)=\sum_{j=1}^{3} A_{j}(x)x_{j}$ and $w\in H^{s}(\R^{3})$ is a positive ground state solution to the autonomous problem \eqref{AP0} (see Lemma \ref{FS}), and
let $t_{\e}>0$ be the unique number such that 
$$
\max_{t\geq 0} J_{\e}(t \Psi_{\e, y})=J_{\e}(t_{\e} \Psi_{\e, y}). 
$$
Finally, we consider $\Phi_{\e}: M\rightarrow \N_{\e}$ defined by setting
$$
\Phi_{\e}(y)= t_{\e} \Psi_{\e, y}.
$$
\begin{lem}\label{lem3.4}
The functional $\Phi_{\e}$ satisfies the following limit
\begin{equation*}
\lim_{\e\rightarrow 0} J_{\e}(\Phi_{\e}(y))=c_{V_{0}} \mbox{ uniformly in } y\in M.
\end{equation*}
\end{lem}
\begin{proof}
Assume by contradiction that there exist $\delta_{0}>0$, $(y_{n})\subset M$ and $\e_{n}\rightarrow 0$ such that 
\begin{equation}\label{puac}
|J_{\e_{n}}(\Phi_{\e_{n}}(y_{n}))-c_{V_{0}}|\geq \delta_{0}.
\end{equation}
Using Lemma $4.1$ in \cite{AD} and the Dominated Convergence Theorem we can observe that 
\begin{align}\begin{split}\label{nio3}
&\| \Psi_{\e_{n}, y_{n}} \|^{2}_{\e_{n}}\rightarrow \|w\|^{2}_{V_{0}}\in (0, \infty). 
%\mbox{ and } \int_{\R^{3}}\frac{f(|t_{\e_{n}}\Psi_{\e_{n},y_{n}}|^{2})}{|t_{\e_{n}}\Psi_{\e_{n}, y_{n}}|^{2}}  |\Psi_{\e_{n}, y_{n}}|^{4}dx\rightarrow \int_{\R^{3}} \frac{f(t^{2}_{0}w^{2})}{(t_{0}w)^{2}} w^{4}dx.
\end{split}\end{align}
%By using the change of variable $\displaystyle{z=\frac{\e_{n}x-y_{n}}{\e_{n}}}$, if $z\in B_{\frac{\delta}{\e_{n}}}(0)\subset M_{\delta}\subset \Lambda$, it follows that $\e_{n} z\in B_{\delta}(0)$ and $\e_{n} x+y_{n}\in B_{\delta}(y_{n})\subset M_{\delta}$. 
%Then, recalling the definitions of $g$ and $\eta$, we have
%\begin{align}\label{HZ}
%J_{\e}(\Phi_{\e_{n}}(y_{n}))&=\frac{t^{2}_{\e_{n}}}{2} \left(\int_{\R^{3}} |(-\Delta)^{\frac{s}{2}} (\eta(|\e_{n} z|) w(z))|^{2} \, dz+\int_{\R^{3}} V(\e_{n} z+y_{n}) (\eta(|\e_{n} z|) w(z))^{2} \, dz \right) \nonumber\\
%&+\frac{t^{4}_{\e_{n}}}{4}\int_{\R^{3}} \phi^{t}_{\eta(|\e_{n}z|)}(\eta(|\e_{n}z|)w(z))^{2}dz-\int_{\R^{3}} F(t_{\e_{n}} \eta(|\e_{n} z|) w(z)) \, dz \nonumber\\
%&- \frac{t_{\e_{n}}^{2^{*}_{s}}}{2^{*}_{s}}\int_{\R^{3}} (\eta(|\e_{n}z|)w(z))^{2^{*}_{s}}dz.
%\end{align}
%Now, we aim to show that the sequence $\{t_{\e_{n}}\}_{n\in \mathbb{N}}$ verifies $t_{\e_{n}}\rightarrow 1$ as $\e_{n}\rightarrow 0$.
On the other hand, since $\langle J'_{\e_{n}}(\Phi_{\e_{n}}(y_{n})),\Phi_{\e_{n}}(y_{n})\rangle=0$ and using the change of variable $\displaystyle{z=\frac{\e_{n}x-y_{n}}{\e_{n}}}$ it follows that
%From the definition of $t_{\e_{n}}$, it follows that  which gives
\begin{align*}
%\label{1nio}
&t_{\e_{n}}^{2}\|\Psi_{\e_{n}, y_{n}}\|_{\e_{n}}^{2} +b t_{\e_{n}}^{4} [\Psi_{\e_{n}, y_{n}}]^{4}_{A_{\e_{n}}} \nonumber \\
&=\int_{\R^{3}} g(\e_{n}z+y_{n}, |t_{\e_{n}}\eta(|\e_{n}z|)w(z)|^{2}) |t_{\e_{n}}\eta(|\e_{n}z|)w(z)|^{2} dz.
\end{align*}
If $z\in B_{\frac{\delta}{\e_{n}}}(0)\subset M_{\delta}\subset \Lambda$, then $\e_{n} z+y_{n}\in B_{\delta}(y_{n})\subset M_{\delta}\subset \Lambda_{\e}$. Since $g(x,t)=f(t)$ for all $x\in \Lambda$ and $\eta(t)=0$ for $t\geq \delta$, we have
\begin{align}\label{1nio}
&t_{\e_{n}}^{2}\|\Psi_{\e_{n}, y_{n}}\|_{\e_{n}}^{2} +b t_{\e_{n}}^{4} [\Psi_{\e_{n}, y_{n}}]^{4}_{A_{\e_{n}}}\nonumber \\
&=\int_{\R^{3}} f(|t_{\e_{n}}\eta(|\e_{n}z|)w(z)|^{2}) |t_{\e_{n}}\eta(|\e_{n}z|)w(z)|^{2}.
\end{align}
In view of $\eta=1$ in $B_{\frac{\delta}{2}}(0)\subset B_{\frac{\delta}{\e_{n}}}(0)$ for all $n$ large enough and \eqref{1nio} we can deduce that
\begin{align}\label{nioo}
\frac{1}{t_{\e_{n}}^{2}} \|\Psi_{\e_{n}, y_{n}}\|_{\e_{n}}^{2}+ b[\Psi_{\e_{n}, y_{n}}]^{4}_{A_{\e_{n}}} 
&=\int_{\R^{3}} \frac{f(|t_{\e_{n}}\Psi_{\e_{n},y_{n}}|^{2})}{|t_{\e_{n}}\Psi_{\e_{n}, y_{n}}|^{2}}  |\Psi_{\e_{n}, y_{n}}|^{4}dx \nonumber \\
&\geq \int_{B_{\frac{\delta}{2}}(0)} \frac{f(|t_{\e_{n}}\eta(|\e_{n}z|)w(z)|^{2})}{|t_{\e_{n}}\eta(|\e_{n}z|)w(z)|^{2}} (\eta(|\e_{n}z|)w(z))^{4}dz \nonumber \\
&=\int_{B_{\frac{\delta}{2}}(0)} \frac{f(|t_{\e_{n}}w(z)|^{2})}{|t_{\e_{n}}w(z)|^{2}} w(z)^{4}dz \nonumber \\
&\geq   \frac{f(|t_{\e_{n}}w(\hat{z})|^{2})}{|t_{\e_{n}}w(\hat{z})|^{2}}w(\hat{z})^{4} |B_{\frac{\delta}{2}}(0)|, 
\end{align}
where
\begin{equation*}
w(\hat{z})=\min_{z\in B_{\frac{\delta}{2}}} w(z)>0.
\end{equation*} 
%so that by using $(f_4)$ we deduce that
%\begin{align}\label{nioo}
%&\frac{1}{t_{\e_{n}}^{2}} \int_{\R^{3}} |(-\Delta)^{\frac{s}{2}}\Psi_{\e_{n}, y_{n}}|^{2}+V(\e_{n}x) \Psi_{\e_{n}, y_{n}}^{2}dx+\int_{\R^{3}} \phi_{\Psi_{\e_{n}, y_{n}}}^{t} \Psi_{\e_{n}, y_{n}}^{2}dx \nonumber \\
%&\geq t_{\e_{n}}^{2^{*}_{s}-2} w^{2^{*}_{s}}(\hat{z}) |B_{\frac{\delta}{2}}(0)|.
%\end{align}
Now, assume by contradiction that $t_{\e_{n}}\rightarrow \infty$. 
This fact, \eqref{nioo} and \eqref{nio3} yield 
$$
b[w]^{4}=\infty,
$$
that is a contradiction.
Hence, $(t_{\e_{n}})$ is bounded and, up to subsequence, we may assume that $t_{\e_{n}}\rightarrow t_{0}$ for some $t_{0}\geq 0$.  
In particular $t_{0}>0$. In fact, if $t_{0}=0$, we can see that \eqref{uNr} and \eqref{1nio} imply that
$$
\min\{a, 1\} r\leq \int_{\R^{3}} f(|t_{\e_{n}}\eta(|\e_{n}z|)w(z)|^{2}) |t_{\e_{n}}\eta(|\e_{n}z|)w(z)|^{2}.
$$
In view of assumptions $(f_1)$-$(f_2)$ and \eqref{nio3} we can deduce that $t_{0}$ can not be zero.
%Otherwise, if $t_{0}=0$, we can use \eqref{nio3}, \eqref{1nio}, assumptions $(g_1)$ and $(g_2)$ and \eqref{uNr} to get an absurd.
%\begin{align*}
%\|t_{\e_{n}} \Psi_{\e_{n}, y_{n}}\|_{\e_{n}}^{2}\rightarrow 0
%\end{align*}
%which is impossible because of \eqref{uNr}. 
Hence $t_{0}>0$.
Thus, letting the limit as $n\rightarrow \infty$ in \eqref{1nio}, we can see that
\begin{align*}
\frac{1}{t_{0}^{2}}\|w\|^{2}_{V_{0}}+b[w]^{4}=\int_{\R^{3}} \frac{f((t_{0} w)^{2})}{(t_{0}w)^{2}} \,w^{4} \, dx.
\end{align*}
Taking into account $w\in \N_{0}$ and using the fact that $\frac{f(t)}{t}$ is increasing by $(f_4)$, we can infer that $t_{0}=1$.
Letting the limit as $n\rightarrow \infty$ and using $t_{\e_{n}}\rightarrow 1$ 
we can conclude that
$$
\lim_{n\rightarrow \infty} J_{\e_{n}}(\Phi_{\e_{n}, y_{n}})=J_{V_{0}}(w)=c_{V_{0}},
$$
which provides a contradiction in view of \eqref{puac}.
\end{proof}

\noindent
%Now, we are in the position to define the barycenter map. 
For any $\delta>0$, we take $\rho=\rho(\delta)>0$ such that $M_{\delta}\subset B_{\rho}$, and we consider $\varUpsilon: \R^{3}\rightarrow \R^{3}$ defined by setting
\begin{equation*}
\varUpsilon(x)=
\left\{
\begin{array}{ll}
x &\mbox{ if } |x|<\rho \\
\frac{\rho x}{|x|} &\mbox{ if } |x|\geq \rho.
\end{array}
\right.
\end{equation*}
We define the barycenter map $\beta_{\e}: \N_{\e}\rightarrow \R^{3}$ as follows
\begin{align*}
\beta_{\e}(u)=\frac{\displaystyle{\int_{\R^{3}} \varUpsilon(\e x)|u(x)|^{4} \,dx}}{\displaystyle{\int_{\R^{3}} |u(x)|^{4} \,dx}}.
\end{align*}

\noindent
Arguing as Lemma $4.3$ in \cite{AD}, it is easy to see that the function $\beta_{\e}$ verifies the following limit:
\begin{lem}\label{lem3.5N}
\begin{equation*}
\lim_{\e \rightarrow 0} \beta_{\e}(\Phi_{\e}(y))=y \mbox{ uniformly in } y\in M.
\end{equation*}
\end{lem}

\noindent
The next compactness result will play a fundamental role to prove that the solutions of \eqref{MPe} are also solution to \eqref{Pe}.
\begin{lem}\label{prop3.3}
%\label{lem3.1N}
Let $\e_{n}\rightarrow 0$ and $(u_{n})\subset \mathcal{N}_{\e_{n}}$ be such that $J_{\e_{n}}(u_{n})\rightarrow c_{V_{0}}$. Then there exists $(\tilde{y}_{n})\subset \R^{3}$ such that $v_{n}(x)=|u_{n}|(x+\tilde{y}_{n})$ has a convergent subsequence in $H^{s}_{V_{0}}$. Moreover, up to a subsequence, $y_{n}=\e_{n} \tilde{y}_{n}\rightarrow y_{0}$ for some $y_{0}\in M$.
\end{lem}
\begin{proof}
Using $\langle J'_{\e_{n}}(u_{n}), u_{n}\rangle=0$, $J_{\e_{n}}(u_{n})= c_{V_{0}}+o_{n}(1)$, Lemma \ref{AMlem1} and arguing as in the first part of Lemma \ref{PSc}, we can see that $\|u_{n}\|_{\e_{n}}\leq C$ for all $n\in \mathbb{N}$. Moreover, from Lemma \ref{DI}, we also know that $(|u_{n}|)$ is bounded in $H^{s}_{V_{0}}$.
Arguing as in the proof of Lemma \ref{LionsFS}, 
we can find a sequence $(\tilde{y}_{n})\subset \R^{3}$, and constants $R>0$ and $\beta>0$ such that
\begin{equation}\label{sacchi}
\liminf_{n\rightarrow \infty}\int_{B_{R}(\tilde{y}_{n})} |u_{n}|^{2} \, dx\geq \beta>0.
\end{equation}
Put $v_{n}(x)=|u_{n}|(x+\tilde{y}_{n})$. Hence, $(v_{n})$ is bounded in $H^{s}_{V_{0}}$ and we may assume that 
$v_{n}\rightharpoonup v\not\equiv 0$ in $H^{s}(\R^{3}, \R)$  as $n\rightarrow \infty$.
Fix $t_{n}>0$ such that $\tilde{v}_{n}=t_{n} v_{n}\in \mathcal{N}_{V_{0}}$. Using Lemma \ref{DI}, we can deduce that 
$$
c_{V_{0}}\leq J_{V_{0}}(\tilde{v}_{n})\leq \max_{t\geq 0}J_{\e_{n}}(tv_{n})= J_{\e_{n}}(u_{n})
%c_{V_{0}}\leq J_{0}(\tilde{v}_{n})\leq \max_{t\geq 0}J_{\e_{n}}(tv_{n})= J_{\e_{n}}(u_{n})= c_{V_{0}}+o_{n}(1)
$$
which together with Lemma \ref{AMlem1} yields $J_{V_{0}}(\tilde{v}_{n})\rightarrow c_{V_{0}}$. Moreover, $\tilde{v}_{n}\nrightarrow 0$ in $H^{s}_{V_{0}}$.
Since $(v_{n})$ and $(\tilde{v}_{n})$ are bounded in $H^{s}_{V_{0}}$ and $\tilde{v}_{n}\nrightarrow 0$  in $H^{s}_{V_{0}}$, we deduce that $t_{n}\rightarrow t^{*}\geq 0$. Indeed $t^{*}>0$ since $\tilde{v}_{n}\nrightarrow 0$  in $H^{s}_{V_{0}}$. From the uniqueness of the weak limit, we can deduce that $\tilde{v}_{n}\rightharpoonup \tilde{v}=t^{*}v\not\equiv 0$ in $H^{s}_{V_{0}}$. 
This combined with Lemma \ref{FS} implies that
\begin{equation}\label{elena}
\tilde{v}_{n}\rightarrow \tilde{v} \mbox{ in } H^{s}_{V_{0}}.
\end{equation} 
Consequently, $v_{n}\rightarrow v$ in $H^{s}_{V_{0}}$ as $n\rightarrow \infty$.

Now, we set $y_{n}=\e_{n}\tilde{y}_{n}$ and we show that $(y_{n})$ admits a subsequence, still denoted by $y_{n}$, such that $y_{n}\rightarrow y_{0}$ for some $y_{0}\in M$. We begin proving that $(y_{n})$ is bounded. Assume by contradiction that, up to a subsequence, $|y_{n}|\rightarrow \infty$ as $n\rightarrow \infty$. Choose $R>0$ such that $\Lambda \subset B_{R}(0)$. Then for $n$ large enough, we have $|y_{n}|>2R$, and for any $z\in B_{R/\e_{n}}$ it holds
$$
|\e_{n}z+y_{n}|\geq |y_{n}|-|\e_{n}z|>R.
$$
Taking into account $(u_{n})\subset \N_{\e_{n}}$, $(V_{1})$, Lemma \ref{DI}, the definition of $g$ and the change of variable $x\mapsto z+\tilde{y}_{n}$, we have
\begin{align*}
%\label{pasq}
a[v_{n}]^{2}+\int_{\R^{3}} V_{0} v_{n}^{2}\, dx &\leq a[v_{n}]^{2}+\int_{\R^{3}} V_{0} v_{n}^{2}\, dx+b[v_{n}]^{4} \nonumber \\
&\leq\int_{\R^{3}} g(\e_{n} x+y_{n}, |v_{n}|^{2}) |v_{n}|^{2} \, dx \nonumber\\
&\leq \int_{B_{\frac{R}{\e_{n}}}(0)} \tilde{f}(|v_{n}|^{2}) |v_{n}|^{2} \, dx+\int_{\R^{3}\setminus B_{\frac{R}{\e_{n}}}(0)} f(|v_{n}|^{2}) |v_{n}|^{2}+|v_{n}|^{\2} \, dx \nonumber \\
&\leq \frac{V_{0}}{k}\int_{\R^{3}}|v_{n}|^{2}\, dx.
\end{align*}
%Then, recalling that $v_{n}\rightarrow v$ in $H^{s}(\R^{3}, \R)$ as $n\rightarrow \infty$ and $\tilde{f}(t)\leq \frac{V_{0}}{k}$, we can see that (\ref{pasq}) yields
which implies that
%$$
%\left(1-\frac{1}{k}\right)\left([v_{n}]^{2}+\int_{\R^{3}} |v_{n}|^{2}\, dx\right)\leq o_{n}(1),
%$$
$v_{n}\rightarrow 0$ in $H^{s}_{V_{0}}$, that is a contradiction. Therefore, $(y_{n})$ is bounded and we may assume that $y_{n}\rightarrow y_{0}\in \R^{3}$. If $y_{0}\notin \overline{\Lambda}$, we can proceed as above to infer that $v_{n}\rightarrow 0$ in $H^{s}_{V_{0}}$, which is impossible. Thus $y_{0}\in \overline{\Lambda}$. 
Now, we aim to prove that $V(y_{0})=V_{0}$. Assume by contradiction that $V(y_{0})>V_{0}$.
In the light of (\ref{elena}), Fatou's Lemma, the invariance of  $\R^{3}$ by translations, Lemma \ref{DI} and Lemma \ref{AMlem1}, we obtain 
\begin{align*}
c_{V_{0}}=J_{V_{0}}(\tilde{v})&<\frac{a}{2}[\tilde{v}]^{2}+\frac{1}{2}\int_{\R^{3}} V(y_{0})\tilde{v}^{2} \, dx+\frac{b}{4}[\tilde{v}]^{4}-\frac{1}{2}\int_{\R^{3}} F(|\tilde{v}|^{2}) dx\\
&\leq \liminf_{n\rightarrow \infty}\Bigl[\frac{a}{2}[\tilde{v}_{n}]^{2}+\frac{1}{2}\int_{\R^{3}} V(\e_{n}x+y_{n}) |\tilde{v}_{n}|^{2} \, dx+\frac{b}{4}[\tilde{v}_{n}]^{4}-\frac{1}{2}\int_{\R^{3}} F(|\tilde{v}_{n}|^{2})\, dx  \Bigr] \\
&\leq \liminf_{n\rightarrow \infty}\Bigl[a\frac{t_{n}^{2}}{2}[|u_{n}|]^{2}+\frac{t_{n}^{2}}{2}\int_{\R^{3}} V(\e_{n}z) |u_{n}|^{2} \, dz
+ b\frac{t_{n}^{4}}{4}[|u_{n}|]^{4}-\frac{1}{2}\int_{\R^{3}} F(|t_{n} u_{n}|^{2}) \, dz  \Bigr] \\
&\leq \liminf_{n\rightarrow \infty} J_{\e_{n}}(t_{n} u_{n}) \leq \liminf_{n\rightarrow \infty} J_{\e_{n}}(u_{n})= c_{V_{0}}
\end{align*}
which is an absurd. Therefore, in view of $(V_2)$, we can conclude that $y_{0}\in M$.
\end{proof}

\noindent
Now, we consider the following subset of $\N_{\e}$ 
%introduce a subset $\widetilde{\N}_{\e}$ of $\N_{\e}$ by taking a function $h_{1}:\R^{+}\rightarrow \R^{+}$ such that $h_{1}(\e)\rightarrow 0$ as $\e \rightarrow 0$, and setting
$$
\widetilde{\N}_{\e}=\left \{u\in \N_{\e}: J_{\e}(u)\leq c_{V_{0}}+h_{1}(\e)\right\},
$$
where $h_{1}:\R^{+}\rightarrow \R^{+}$ is such that $h_{1}(\e)\rightarrow 0$ as $\e \rightarrow 0$.
Fixed $y\in M$, we can use Lemma \ref{lem3.4} to see that $h_{1}(\e)=|J_{\e}(\Phi_{\e}(y))-c_{V_{0}}|\rightarrow 0$ as $\e \rightarrow 0$. Therefore $\Phi_{\e}(y)\in \widetilde{\N}_{\e}$, and $\widetilde{\N}_{\e}\neq \emptyset$ for any $\e>0$. Arguing as in Lemma $4.5$ in \cite{AD}, we have:
\begin{lem}\label{lem3.5}
For any $\delta>0$, there holds that
$$
\lim_{\e \rightarrow 0} \sup_{u\in \widetilde{\mathcal{N}}_{\e}} {\rm dist}(\beta_{\e}(u), M_{\delta})=0.
$$
\end{lem}

\noindent
We end this section proving a multiplicity result for \eqref{MPe}.
\begin{thm}\label{multiple}
For any $\delta>0$ such that $M_{\delta}\subset \Lambda$, there exists $\tilde{\e}_{\delta}>0$ such that, for any $\e\in (0, \e_{\delta})$, problem \eqref{MPe} has at least $cat_{M_{\delta}}(M)$ nontrivial solutions.
\end{thm}
\begin{proof}
Given  $\delta>0$ such that $M_{\delta}\subset \Lambda$, we can use Lemma \ref{lem3.5N}, Lemma \ref{lem3.4}, Lemma \ref{lem3.5} and argue as in \cite{CL} to deduce the existence of $\tilde{\e}_{\delta}>0$ such that, for any $\e\in (0, \e_{\delta})$, the following diagram
$$
M \stackrel{\Phi_{\e}}\rightarrow \widetilde{\mathcal{N}}_{\e} \stackrel{\beta_{\e}}\rightarrow M_{\delta}
$$
is well defined and $\beta_{\e}\circ \Phi_{\e}$ is homotopically equivalent to the embedding  $\iota: M\rightarrow M_{\delta}$. Hence, $cat_{\widetilde{\mathcal{N}}_{\e}}(\widetilde{\mathcal{N}}_{\e})\geq cat_{M_{\delta}}(M)$.
From Proposition \ref{propPSc} and standard Ljusternik-Schnirelmann theory, we can deduce that $J_{\e}$  possesses at least $cat_{\widetilde{\mathcal{N}}_{\e}}(\widetilde{\mathcal{N}}_{\e})$  critical points on $\mathcal{N}_{\e}$. In view of Corollary \ref{cor},  we obtain $cat_{M_{\delta}}(M)$  nontrivial solutions for  \eqref{MPe}.
\end{proof}

\section{Proof of Theorem \ref{thm1}}
This last section is devoted to the proof of the main result of this paper. In order to show that the solutions of \eqref{MPe} are indeed solutions to \eqref{Pe}, we need to verify that \eqref{ue} holds true.
For this purpose, we begin proving the following fundamental result in which we use a variant of Moser iteration scheme \cite{Moser} and a Kato's approximation argument \cite{Kato}. 
%which will be the main key to deduce that the solutions to \eqref{Pe} are indeed solutions  to \eqref{P}.
\begin{lem}\label{moser} 
Let $\e_{n}\rightarrow 0$ and $u_{n}\in \widetilde{\mathcal{N}}_{\e_{n}}$ be a solution to \eqref{MPe}. 
Then $v_{n}=|u_{n}|(\cdot+\tilde{y}_{n})$ satisfies $v_{n}\in L^{\infty}(\R^{3},\R)$ and there exists $C>0$ such that 
$$
\|v_{n}\|_{L^{\infty}(\R^{3})}\leq C \mbox{ for all } n\in \mathbb{N},
$$
where $\tilde{y}_{n}$ is given by Lemma \ref{prop3.3}.
Moreover
$$
\lim_{|x|\rightarrow \infty} v_{n}(x)=0 \mbox{ uniformly in } n\in \mathbb{N}.
$$
\end{lem}
\begin{proof}
For any $L>0$ we define $u_{L,n}:=\min\{|u_{n}|, L\}\geq 0$ and we set $v_{L, n}=u_{L,n}^{2(\beta-1)}u_{n}$, where $\beta>1$ will be chosen later.
Taking $v_{L, n}$ as test function in (\ref{MPe}) we can see that
\begin{align}\label{conto1FF}
&(a+b[u_{n}]^{2}_{A_{\e_{n}}})\Re\left(\iint_{\R^{6}} \frac{(u_{n}(x)-u_{n}(y)e^{\imath A(\frac{x+y}{2})\cdot (x-y)})}{|x-y|^{3+2s}} \overline{(u_{n}u_{L,n}^{2(\beta-1)}(x)-u_{n}u_{L,n}^{2(\beta-1)}(y)e^{\imath A(\frac{x+y}{2})\cdot (x-y)})} \, dx dy\right)   \nonumber \\
&=\int_{\R^{3}} g(\e_{n} x, |u_{n}|^{2}) |u_{n}|^{2}u_{L,n}^{2(\beta-1)}  \,dx-\int_{\R^{3}} V(\e_{n} x) |u_{n}|^{2} u_{L,n}^{2(\beta-1)} \, dx.
\end{align}
Now, we observe that
\begin{align*}
&\Re\left[(u_{n}(x)-u_{n}(y)e^{\imath A(\frac{x+y}{2})\cdot (x-y)})\overline{(u_{n}u_{L,n}^{2(\beta-1)}(x)-u_{n}u_{L,n}^{2(\beta-1)}(y)e^{\imath A(\frac{x+y}{2})\cdot (x-y)})}\right] \\
&=\Re\Bigl[|u_{n}(x)|^{2}v_{L}^{2(\beta-1)}(x)-u_{n}(x)\overline{u_{n}(y)} u_{L,n}^{2(\beta-1)}(y)e^{-\imath A(\frac{x+y}{2})\cdot (x-y)}-u_{n}(y)\overline{u_{n}(x)} u_{L,n}^{2(\beta-1)}(x) e^{\imath A(\frac{x+y}{2})\cdot (x-y)} \\
&+|u_{n}(y)|^{2}u_{L,n}^{2(\beta-1)}(y) \Bigr] \\
&\geq (|u_{n}(x)|^{2}u_{L,n}^{2(\beta-1)}(x)-|u_{n}(x)||u_{n}(y)|u_{L,n}^{2(\beta-1)}(y)-|u_{n}(y)||u_{n}(x)|u_{L,n}^{2(\beta-1)}(x)+|u_{n}(y)|^{2}u^{2(\beta-1)}_{L,n}(y) \\
&=(|u_{n}(x)|-|u_{n}(y)|)(|u_{n}(x)|u_{L,n}^{2(\beta-1)}(x)-|u_{n}(y)|u_{L,n}^{2(\beta-1)}(y)).
\end{align*}
Thus
\begin{align}\label{realeF}
&\Re\left(\iint_{\R^{6}} \frac{(u_{n}(x)-u_{n}(y)e^{\imath A(\frac{x+y}{2})\cdot (x-y)})}{|x-y|^{3+2s}} \overline{(u_{n}u_{L,n}^{2(\beta-1)}(x)-u_{n}u_{L,n}^{2(\beta-1)}(y)e^{\imath A(\frac{x+y}{2})\cdot (x-y)})} \, dx dy\right) \nonumber\\
&\geq \iint_{\R^{6}} \frac{(|u_{n}(x)|-|u_{n}(y)|)}{|x-y|^{3+2s}} (|u_{n}(x)|u_{L,n}^{2(\beta-1)}(x)-|u_{n}(y)|u_{L,n}^{2(\beta-1)}(y))\, dx dy.
\end{align}
For all $t\geq 0$, we define
\begin{equation*}
\gamma(t)=\gamma_{L, \beta}(t)=t t_{L}^{2(\beta-1)},
\end{equation*}
where  $t_{L}=\min\{t, L\}$. 
Since $\gamma$ is an increasing function, we have
\begin{align*}
(p-q)(\gamma(p)- \gamma(q))\geq 0 \quad \mbox{ for any } p, q\in \R.
\end{align*}
Let us consider the functions 
\begin{equation*}
\Lambda(t)=\frac{|t|^{2}}{2} \quad \mbox{ and } \quad \Gamma(t)=\int_{0}^{t} (\gamma'(\tau))^{\frac{1}{2}} d\tau. 
\end{equation*}
and we note that
\begin{equation}\label{Gg}
\Lambda'(p-q)(\gamma(p)-\gamma(q))\geq |\Gamma(p)-\Gamma(q)|^{2} \quad \mbox{ for any } p, q\in\R. 
\end{equation}
Indeed, for any $p, q\in \R$ such that $p<q$, the Jensen inequality yields
\begin{align*}
\Lambda'(p-q)(\gamma(p)-\gamma(q))&=(p-q)\int_{q}^{p} \gamma'(t)dt\\
&=(p-q)\int_{q}^{p} (\Gamma'(t))^{2}dt \\
&\geq \left(\int_{q}^{p} \Gamma'(t) dt\right)^{2}\\
&=(\Gamma(p)-\Gamma(q))^{2}.
\end{align*}
In a similar way, we can prove that if $p\geq q$ then $\Lambda'(p-q)(\gamma(p)-\gamma(q))\geq (\Gamma(q)-\Gamma(p))^{2}$, that is \eqref{Gg} holds.
Hence, in view of \eqref{Gg}, we can deduce that
\begin{align}\label{Gg1}
|\Gamma(|u_{n}(x)|)- \Gamma(|u_{n}(y)|)|^{2} \leq (|u_{n}(x)|- |u_{n}(y)|)((|u_{n}|u_{L,n}^{2(\beta-1)})(x)- (|u_{n}|u_{L,n}^{2(\beta-1)})(y)). 
\end{align}
By \eqref{realeF} and \eqref{Gg1}, it follows that
\begin{align}\label{conto1FFF}
&\Re\left(\iint_{\R^{6}} \frac{(u_{n}(x)-u_{n}(y)e^{\imath A(\frac{x+y}{2})\cdot (x-y)})}{|x-y|^{3+2s}} \overline{(u_{n}u_{L,n}^{2(\beta-1)}(x)-u_{n}u_{L,n}^{2(\beta-1)}(y)e^{\imath A(\frac{x+y}{2})\cdot (x-y)})} \, dx dy\right) \nonumber \\
&\geq [\Gamma(|u_{n}|)]^{2}.
\end{align}
%\begin{align}\label{BMS}
%&[\Gamma(v_{n})]^{2}+\int_{\R^{N}} V_{0}|v_{n}|^{2}v_{L, n}^{2(\beta-1)} dx \nonumber \\
%&\leq \iint_{\R^{2N}} \frac{(v_{n}(x)- v_{n}(y))}{|x-y|^{N+2s}} ((v_{n}v_{L, n}^{2(\beta-1)})(x)-(v_{n} v_{L,n}^{2(\beta-1)})(y)) \,dx dy +\int_{\R^{N}} V(\e_{n}x+\e_{n}\tilde{y}_{n})|v_{n}|^{2}v_{L,n}^{2(\beta-1)} dx \nonumber\\
%&\leq \int_{\R^{N}} g(\e_{n}x+\e_{n}\tilde{y}_{n}, v^{2}_{n}) v^{2}_{n} v_{L,n}^{2(\beta-1)} dx.
%\end{align}
Observing that $\Gamma(|u_{n}|)\geq \frac{1}{\beta} |u_{n}| u_{L,n}^{\beta-1}$ and recalling the fractional Sobolev embedding $\mathcal{D}^{s,2}(\R^{3}, \R)\subset L^{\2}(\R^{3}, \R)$ (see \cite{DPV}), we get 
\begin{equation}\label{SS1}
[\Gamma(|u_{n}|)]^{2}\geq S_{*} \|\Gamma(|u_{n}|)\|^{2}_{L^{\2}(\R^{3})}\geq \left(\frac{1}{\beta}\right)^{2} S_{*}\||u_{n}| u_{L,n}^{\beta-1}\|^{2}_{L^{\2}(\R^{3})}.
\end{equation}
Putting together \eqref{conto1FF}, \eqref{conto1FFF}, \eqref{SS1} and noting that $a\leq a+b[u_{n}]_{A_{\e_{n}}}^{2}\leq a+bM^{2}$, we obtain that
\begin{align}\label{BMS}
a\left(\frac{1}{\beta}\right)^{2} S_{*}\||u_{n}| u_{L,n}^{\beta-1}\|^{2}_{L^{\2}(\R^{3})}+\int_{\R^{3}} V(\e_{n} x)|u_{n}|^{2}u_{L,n}^{2(\beta-1)} dx\leq \int_{\R^{3}} g(\e_{n}x, |u_{n}|^{2}) |u_{n}|^{2} u_{L,n}^{2(\beta-1)} dx.
\end{align}
Now, by $(g_1)$ and $(g_2)$, it follows that for any $\xi>0$ there exists $C_{\xi}>0$ such that
\begin{equation}\label{SS2}
g(\e_{n} x, t^{2})t^{2}\leq \xi |t|^{2}+C_{\xi}|t|^{\2} \mbox{ for all } t\in \R.
\end{equation}
Taking $\xi\in (0, V_{0})$ and using \eqref{BMS} and \eqref{SS2} we have
\begin{equation}\label{simo1}
\|w_{L,n}\|_{L^{\2}(\R^{3})}^{2}\leq C \beta^{2} \int_{\R^{3}} |u_{n}|^{\2}u_{L,n}^{2(\beta-1)},
\end{equation}
where we set $w_{L,n}:=|u_{n}| u_{L,n}^{\beta-1}$.

Take $\beta=\frac{\2}{2}$ and fix $R>0$. Recalling that $0\leq u_{L,n}\leq |u_{n}|$ and applying H\"older inequality, we get
\begin{align}\label{simo2}
\int_{\R^{3}} |u_{n}|^{\2}u_{L,n}^{2(\beta-1)}dx&=\int_{\R^{3}} |u_{n}|^{\2-2} |u_{n}|^{2} u_{L,n}^{\2-2}dx \nonumber\\
&=\int_{\R^{3}} |u_{n}|^{\2-2} (|u_{n}| u_{L,n}^{\frac{\2-2}{2}})^{2}dx \nonumber\\
&\leq \int_{\{|u_{n}|<R\}} R^{\2-2} |u_{n}|^{\2} dx+\int_{\{|u_{n}|>R\}} |u_{n}|^{\2-2} (|u_{n}| u_{L,n}^{\frac{\2-2}{2}})^{2}dx \nonumber\\
&\leq \int_{\{|u_{n}|<R\}} R^{\2-2} |u_{n}|^{\2} dx+\left(\int_{\{|u_{n}|>R\}} |u_{n}|^{\2} dx\right)^{\frac{\2-2}{\2}} \left(\int_{\R^{3}} (|u_{n}| u_{L,n}^{\frac{\2-2}{2}})^{\2}dx\right)^{\frac{2}{\2}}.
\end{align}
Since $(|u_{n}|)$ is bounded in $H^{s}(\R^{3}, \R)$, we can see that for any $R$ sufficiently large
\begin{equation}\label{simo3}
\left(\int_{\{|u_{n}|>R\}} |u_{n}|^{\2} dx\right)^{\frac{\2-2}{\2}}\leq  \frac{1}{2\beta^{2}}.
\end{equation}
In the light of \eqref{simo1}, \eqref{simo2} and \eqref{simo3}, we infer that
\begin{equation*}
\left(\int_{\R^{3}} (|u_{n}| u_{L,n}^{\frac{\2-2}{2}})^{\2} \right)^{\frac{2}{\2}}\leq C\beta^{2} \int_{\R^{3}} R^{\2-2} |u_{n}|^{\2} dx<\infty,
\end{equation*}
and taking the limit as $L\rightarrow \infty$ we deduce that $|u_{n}|\in L^{\frac{(\2)^{2}}{2}}(\R^{3},\R)$.

Using $0\leq u_{L,n}\leq |u_{n}|$ and passing to the limit as $L\rightarrow \infty$ in \eqref{simo1}, we have
\begin{equation*}
\|u_{n}\|_{L^{\beta\2}(\R^{3})}^{2\beta}\leq C \beta^{2} \int_{\R^{3}} |u_{n}|^{\2+2(\beta-1)},
\end{equation*}
from which we deduce that
\begin{equation*}
\left(\int_{\R^{3}} |u_{n}|^{\beta\2} dx\right)^{\frac{1}{(\beta-1)\2}}\leq C \beta^{\frac{1}{\beta-1}} \left(\int_{\R^{3}} |u_{n}|^{\2+2(\beta-1)}\right)^{\frac{1}{2(\beta-1)}}.
\end{equation*}
For $m\geq 1$ we define $\beta_{m+1}$ inductively so that $\2+2(\beta_{m+1}-1)=\2 \beta_{m}$ and $\beta_{1}=\frac{\2}{2}$. \\
Therefore
\begin{equation*}
\left(\int_{\R^{3}} |u_{n}|^{\beta_{m+1}\2} dx\right)^{\frac{1}{(\beta_{m+1}-1)\2}}\leq C \beta_{m+1}^{\frac{1}{\beta_{m+1}-1}} \left(\int_{\R^{3}} |u_{n}|^{\2\beta_{m}}\right)^{\frac{1}{\2(\beta_{m}-1)}}.
\end{equation*}
Let us define
$$
D_{m}=\left(\int_{\R^{3}} |u_{n}|^{\2\beta_{m}}\right)^{\frac{1}{\2(\beta_{m}-1)}}.
$$
Using an iteration argument, we can find $C_{0}>0$ independent of $m$ such that 
$$
D_{m+1}\leq \prod_{k=1}^{m} C \beta_{k+1}^{\frac{1}{\beta_{k+1}-1}}  D_{1}\leq C_{0} D_{1}.
$$
Taking the limit as $m\rightarrow \infty$ we get 
\begin{equation}\label{UBu}
\|u_{n}\|_{L^{\infty}(\R^{3})}\leq C_{0}D_{1}=:K \mbox{ for all } n\in \mathbb{N}.
\end{equation}
Moreover, by interpolation, we can deduce that $(|u_{n}|)$ strongly converges in $L^{r}(\R^{3}, \R)$ for all $r\in (2, \infty)$. From the growth assumptions on $g$, we can see that $g(\e x, |u_{n}|^{2})|u_{n}|$ strongly converges in $L^{r}(\R^{3}, \R)$ for all $r\in [2, \infty)$.

In what follows, we prove that $|u_{n}|$ is a weak subsolution to 
\begin{equation}\label{Kato0}
\left\{
\begin{array}{ll}
(a+b[v]^{2})(-\Delta)^{s}v+V_{0} v=g(\e_{n} x, v^{2})v &\mbox{ in } \R^{3} \\
v\geq 0 \quad \mbox{ in } \R^{3}.
\end{array}
\right.
\end{equation}
Roughly speaking, we will prove a Kato's inequality for the modulus of solutions of \eqref{MPe}.

Fix $\varphi\in C^{\infty}_{c}(\R^{3}, \R)$ such that $\varphi\geq 0$, and we take $\psi_{\delta, n}=\frac{u_{n}}{u_{\delta, n}}\varphi$ as test function in \eqref{Pe}, where $u_{\delta,n}=\sqrt{|u_{n}|^{2}+\delta^{2}}$ for $\delta>0$. Note that $\psi_{\delta, n}\in H^{s}_{\e_{n}}$ for all $\delta>0$ and $n\in \mathbb{N}$. Indeed $\int_{\R^{3}} V(\e_{n} x) |\psi_{\delta,n}|^{2} dx\leq \int_{\supp(\varphi)} V(\e_{n} x)\varphi^{2} dx<\infty$. 
%For simplicity, we write $u$, $u_{\delta}$ and $\psi_{\delta}$ instead of $u_{n}$, $u_{\delta, n}$ and $\psi_{\delta,n}$. 
Now, we can see that 
\begin{align*}
\psi_{\delta,n}(x)-\psi_{\delta,n}(y)e^{\imath A_{\e}(\frac{x+y}{2})\cdot (x-y)}&=\left(\frac{u_{n}(x)}{u_{\delta,n}(x)}\right)\varphi(x)-\left(\frac{u_{n}(y)}{u_{\delta,n}(y)}\right)\varphi(y)e^{\imath A_{\e}(\frac{x+y}{2})\cdot (x-y)}\\
&=\left[\left(\frac{u_{n}(x)}{u_{\delta,n}(x)}\right)-\left(\frac{u_{n}(y)}{u_{\delta,n}(x)}\right)e^{\imath A_{\e}(\frac{x+y}{2})\cdot (x-y)}\right]\varphi(x) \\
&+\left[\varphi(x)-\varphi(y)\right] \left(\frac{u_{n}(y)}{u_{\delta,n}(x)}\right) e^{\imath A_{\e}(\frac{x+y}{2})\cdot (x-y)} \\
&+\left(\frac{u_{n}(y)}{u_{\delta,n}(x)}-\frac{u_{n}(y)}{u_{\delta,n}(y)}\right)\varphi(y) e^{\imath A_{\e}(\frac{x+y}{2})\cdot (x-y)},
\end{align*}
and using $|z+w+k|^{2}\leq 4(|z|^{2}+|w|^{2}+|k|^{2})$ for all $z,w,k\in \C$, $|e^{\imath t}|=1$ for all $t\in \R$, $u_{\delta,n}\geq \delta$, $|\frac{u_{n}}{u_{\delta,n}}|\leq 1$, \eqref{UBu} and $|\sqrt{|z|^{2}+\delta^{2}}-\sqrt{|w|^{2}+\delta^{2}}|\leq ||z|-|w||$ for all $z, w\in \C$, we can deduce that
\begin{align*}
&|\psi_{\delta,n}(x)-\psi_{\delta,n}(y)e^{\imath A_{\e}(\frac{x+y}{2})\cdot (x-y)}|^{2} \\
&\leq \frac{4}{\delta^{2}}|u_{n}(x)-u_{n}(y)e^{\imath A_{\e}(\frac{x+y}{2})\cdot (x-y)}|^{2}\|\varphi\|^{2}_{L^{\infty}(\R^{3})} +\frac{4}{\delta^{2}}|\varphi(x)-\varphi(y)|^{2} \||u_{n}|\|^{2}_{L^{\infty}(\R^{3})} \\
&+\frac{4}{\delta^{4}} \||u_{n}|\|^{2}_{L^{\infty}(\R^{3})} \|\varphi\|^{2}_{L^{\infty}(\R^{3})} |u_{\delta,n}(y)-u_{\delta,n}(x)|^{2} \\
&\leq \frac{4}{\delta^{2}}|u_{n}(x)-u_{n}(y)e^{\imath A_{\e}(\frac{x+y}{2})\cdot (x-y)}|^{2}\|\varphi\|^{2}_{L^{\infty}(\R^{3})} +\frac{4K^{2}}{\delta^{2}}|\varphi(x)-\varphi(y)|^{2} \\
&+\frac{4K^{2}}{\delta^{4}} \|\varphi\|^{2}_{L^{\infty}(\R^{3})} ||u_{n}(y)|-|u_{n}(x)||^{2}. 
\end{align*}
In view of $u_{n}\in H^{s}_{\e_{n}}$, $|u_{n}|\in H^{s}(\R^{3}, \R)$ (by Lemma \ref{DI}) and $\varphi\in C^{\infty}_{c}(\R^{3}, \R)$, we get $\psi_{\delta,n}\in H^{s}_{\e_{n}}$.

Therefore
\begin{align}\label{Kato1}
&(a+b[u_{n}]^{2}_{A_{\e_{n}}})\Re\left[\iint_{\R^{6}} \frac{(u_{n}(x)-u_{n}(y)e^{\imath A_{\e}(\frac{x+y}{2})\cdot (x-y)})}{|x-y|^{3+2s}} \left(\frac{\overline{u_{n}(x)}}{u_{\delta,n}(x)}\varphi(x)-\frac{\overline{u_{n}(y)}}{u_{\delta,n}(y)}\varphi(y)e^{-\imath A_{\e}(\frac{x+y}{2})\cdot (x-y)}  \right) dx dy\right] \nonumber\\
&+\int_{\R^{3}} V(\e x)\frac{|u_{n}|^{2}}{u_{\delta,n}}\varphi dx=\int_{\R^{3}} g(\e x, |u_{n}|^{2})\frac{|u_{n}|^{2}}{u_{\delta,n}}\varphi dx.
\end{align}
From $\Re(z)\leq |z|$ for all $z\in \C$ and  $|e^{\imath t}|=1$ for all $t\in \R$, it follows that
\begin{align}\label{alves1}
&\Re\left[(u_{n}(x)-u_{n}(y)e^{\imath A_{\e}(\frac{x+y}{2})\cdot (x-y)}) \left(\frac{\overline{u_{n}(x)}}{u_{\delta,n}(x)}\varphi(x)-\frac{\overline{u_{n}(y)}}{u_{\delta,n}(y)}\varphi(y)e^{-\imath A_{\e}(\frac{x+y}{2})\cdot (x-y)}  \right)\right] \nonumber\\
&=\Re\left[\frac{|u_{n}(x)|^{2}}{u_{\delta,n}(x)}\varphi(x)+\frac{|u_{n}(y)|^{2}}{u_{\delta,n}(y)}\varphi(y)-\frac{u_{n}(x)\overline{u_{n}(y)}}{u_{\delta,n}(y)}\varphi(y)e^{-\imath A_{\e}(\frac{x+y}{2})\cdot (x-y)} -\frac{u_{n}(y)\overline{u_{n}(x)}}{u_{\delta,n}(x)}\varphi(x)e^{\imath A_{\e}(\frac{x+y}{2})\cdot (x-y)}\right] \nonumber \\
&\geq \left[\frac{|u_{n}(x)|^{2}}{u_{\delta,n}(x)}\varphi(x)+\frac{|u_{n}(y)|^{2}}{u_{\delta,n}(y)}\varphi(y)-|u_{n}(x)|\frac{|u_{n}(y)|}{u_{\delta,n}(y)}\varphi(y)-|u_{n}(y)|\frac{|u_{n}(x)|}{u_{\delta,n}(x)}\varphi(x) \right].
\end{align}
Now, we can note that
\begin{align}\label{alves2}
&\frac{|u_{n}(x)|^{2}}{u_{\delta,n}(x)}\varphi(x)+\frac{|u_{n}(y)|^{2}}{u_{\delta,n}(y)}\varphi(y)-|u_{n}(x)|\frac{|u_{n}(y)|}{u_{\delta,n}(y)}\varphi(y)-|u_{n}(y)|\frac{|u_{n}(x)|}{u_{\delta,n}(x)}\varphi(x) \nonumber\\
&=  \frac{|u_{n}(x)|}{u_{\delta,n}(x)}(|u_{n}(x)|-|u_{n}(y)|)\varphi(x)-\frac{|u_{n}(y)|}{u_{\delta,n}(y)}(|u_{n}(x)|-|u_{n}(y)|)\varphi(y) \nonumber\\
&=\left[\frac{|u_{n}(x)|}{u_{\delta,n}(x)}(|u_{n}(x)|-|u_{n}(y)|)\varphi(x)-\frac{|u_{n}(x)|}{u_{\delta,n}(x)}(|u_{n}(x)|-|u_{n}(y)|)\varphi(y)\right] \nonumber\\
&+\left(\frac{|u_{n}(x)|}{u_{\delta,n}(x)}-\frac{|u_{n}(y)|}{u_{\delta,n}(y)} \right) (|u_{n}(x)|-|u_{n}(y)|)\varphi(y) \nonumber\\
&=\frac{|u_{n}(x)|}{u_{\delta,n}(x)}(|u_{n}(x)|-|u_{n}(y)|)(\varphi(x)-\varphi(y)) +\left(\frac{|u_{n}(x)|}{u_{\delta,n}(x)}-\frac{|u_{n}(y)|}{u_{\delta,n}(y)} \right) (|u_{n}(x)|-|u_{n}(y)|)\varphi(y) \nonumber\\
&\geq \frac{|u_{n}(x)|}{u_{\delta,n}(x)}(|u_{n}(x)|-|u_{n}(y)|)(\varphi(x)-\varphi(y)) 
\end{align}
where in the last inequality we used that
$$
\left(\frac{|u_{n}(x)|}{u_{\delta,n}(x)}-\frac{|u_{n}(y)|}{u_{\delta,n}(y)} \right) (|u_{n}(x)|-|u_{n}(y)|)\varphi(y)\geq 0
$$
due to
$$
h(t)=\frac{t}{\sqrt{t^{2}+\delta^{2}}} \mbox{ is increasing for } t\geq 0 \quad \mbox{ and } \quad \varphi\geq 0 \mbox{ in }\R^{3}.
$$
Observing that
$$
\frac{|\frac{|u_{n}(x)|}{u_{\delta,n}(x)}(|u_{n}(x)|-|u_{n}(y)|)(\varphi(x)-\varphi(y))|}{|x-y|^{3+2s}}\leq \frac{||u_{n}(x)|-|u_{n}(y)||}{|x-y|^{\frac{3+2s}{2}}} \frac{|\varphi(x)-\varphi(y)|}{|x-y|^{\frac{3+2s}{2}}}\in L^{1}(\R^{6}),
$$
and $\frac{|u_{n}(x)|}{u_{\delta,n}(x)}\rightarrow 1$ a.e. in $\R^{3}$ as $\delta\rightarrow 0$,
we can apply \eqref{alves1}, \eqref{alves2} and the Dominated Convergence Theorem to infer that
\begin{align}\label{Kato2}
&\limsup_{\delta\rightarrow 0} \Re\left[\iint_{\R^{6}} \frac{(u_{n}(x)-u_{n}(y)e^{\imath A_{\e}(\frac{x+y}{2})\cdot (x-y)})}{|x-y|^{3+2s}} \left(\frac{\overline{u_{n}(x)}}{u_{\delta,n}(x)}\varphi(x)-\frac{\overline{u_{n}(y)}}{u_{\delta,n}(y)}\varphi(y)e^{-\imath A_{\e}(\frac{x+y}{2})\cdot (x-y)}  \right) dx dy\right] \nonumber\\
&\geq \limsup_{\delta\rightarrow 0} \iint_{\R^{6}} \frac{|u_{n}(x)|}{u_{\delta,n}(x)}(|u_{n}(x)|-|u_{n}(y)|)(\varphi(x)-\varphi(y)) \frac{dx dy}{|x-y|^{3+2s}} \nonumber\\
&=\iint_{\R^{6}} \frac{(|u_{n}(x)|-|u_{n}(y)|)(\varphi(x)-\varphi(y))}{|x-y|^{3+2s}} dx dy.
\end{align}
On the other hand, from the Dominated Convergence Theorem again (we recall that $\frac{|u_{n}|^{2}}{u_{\delta, n}}\leq |u_{n}|$ and $\varphi\in C^{\infty}_{c}(\R^{3}, \R)$) we can see that
\begin{equation}\label{Kato3}
\lim_{\delta\rightarrow 0} \int_{\R^{3}} V(\e_{n} x)\frac{|u_{n}|^{2}}{u_{\delta,n}}\varphi dx=\int_{\R^{3}} V(\e_{n} x)|u_{n}|\varphi dx\geq \int_{\R^{3}} V_{0}|u_{n}|\varphi dx
\end{equation}
%\begin{equation}\label{KatoP}
%\liminf_{\delta\rightarrow 0} \int_{\R^{3}} \phi_{|u_{n}|}^{t} \frac{|u_{n}|^{2}}{u_{\delta,n}}\varphi dx\geq \int_{\R^{3}} \phi_{|u|}^{t} |u|\varphi dx\geq 0
%\end{equation}
and
\begin{equation}\label{Kato4}
\lim_{\delta\rightarrow 0}  \int_{\R^{3}} g(\e_{n} x, |u_{n}|^{2})\frac{|u_{n}|^{2}}{u_{\delta,n}}\varphi dx=\int_{\R^{3}} g(\e_{n} x, |u_{n}|^{2}) |u_{n}|\varphi dx.
\end{equation}
By Lemma \ref{DI} we can also see that
\begin{equation}\label{Kff}
\limsup_{n\rightarrow \infty} (a+b[u_{n}]^{2}_{A_{\e_{n}}})\geq (a+b[|u_{n}|]^{2}).
\end{equation}
Putting together \eqref{Kato1}, \eqref{Kato2}, \eqref{Kato3}, \eqref{Kato4} and \eqref{Kff} we can deduce that
\begin{align*}
(a+b[|u_{n}|]^{2})\iint_{\R^{6}} \frac{(|u_{n}(x)|-|u_{n}(y)|)(\varphi(x)-\varphi(y))}{|x-y|^{3+2s}} dx dy+\int_{\R^{3}} V_{0}|u_{n}|\varphi dx\leq 
\int_{\R^{3}} g(\e_{n} x, |u_{n}|^{2}) |u_{n}|\varphi dx
\end{align*}
for any $\varphi\in C^{\infty}_{c}(\R^{3}, \R)$ such that $\varphi\geq 0$, that is $|u_{n}|$ is a weak subsolution to \eqref{Kato0}.

%Then, using $(V_{1})$, it is clear that $v_{n}=|u_{n}|(\cdot+\tilde{y}_{n})$ solves 
Now, we set $v_{n}=|u_{n}|(\cdot+\tilde{y}_{n})$. Then Lemma \ref{DI} yields
$$
a+b[v_{n}]^{2}=a+b[|u_{n}|]^{2}\leq a+b[u_{n}]_{A_{\e_{n}}}^{2}\leq a+bM^{2}.
$$
We also note that $v_{n}$ satisfies 
\begin{equation}\label{Pkat}
(-\Delta)^{s} v_{n} + \frac{V_{0}}{a+bM^{2}} v_{n}\leq g_{n} \mbox{ in } \R^{3},
\end{equation}
where
$$
g_{n}:=(a+b[v_{n}]^{2})^{-1}[g(\e_{n} x+\e_{n}\tilde{y}_{n}, v_{n}^{2})v_{n}-V_{0}v_{n}]+\frac{V_{0}}{a+bM^{2}}v_{n}.
$$
Let $z_{n}\in H^{s}(\R^{3}, \R)$ be the unique solution to
\begin{equation}\label{US}
(-\Delta)^{s} z_{n} + \frac{V_{0}}{a+bM^{2}} z_{n}=g_{n} \mbox{ in } \R^{3}.
\end{equation}
%where
%$$
%g_{n}:=g(\e_{n} x+\e_{n}\tilde{y}_{n}, v_{n}^{2})v_{n}\in L^{r}(\R^{3}, \R) \quad \forall r\in [2, \infty].
%$$
In the light of \eqref{UBu}, we know that $\|v_{n}\|_{L^{\infty}(\R^{3})}\leq C$ for all $n\in \mathbb{N}$, and by interpolation $v_{n}\rightarrow v$ strongly converges in $L^{r}(\R^{3}, \R)$ for all $r\in (2, \infty)$, for some $v\in L^{r}(\R^{3}, \R)$.
From the growth assumptions on $g$, we can see that 
$$
g_{n}\rightarrow  (a+b[v]^{2})^{-1}[f(v^{2})v-V_{0}v]+\frac{V_{0}}{a+bM^{2}}v \mbox{ in } L^{r}(\R^{3}, \R) \quad \forall r\in [2, \infty),
$$ 
and there exists $C>0$ such that $\|g_{n}\|_{L^{\infty}(\R^{3})}\leq C$ for all $n\in \mathbb{N}$.
Then $z_{n}=\mathcal{K}*g_{n}$ (see \cite{FQT}), where $\mathcal{K}$ is the Bessel kernel, and arguing as in \cite{AM}, we can prove that $|z_{n}(x)|\rightarrow 0$ as $|x|\rightarrow \infty$ uniformly with respect to $n\in \mathbb{N}$.
Taking into account $v_{n}$ satisfies \eqref{Pkat} and $z_{n}$ solves \eqref{US}, we can use a comparison argument to see that $0\leq v_{n}\leq z_{n}$ a.e. in $\R^{3}$ and for all $n\in \mathbb{N}$. In conclusion, $v_{n}(x)\rightarrow 0$ as $|x|\rightarrow \infty$ uniformly with respect to $n\in \mathbb{N}$.
\end{proof}

\noindent
Now, we are ready to give the proof of Theorem \ref{thm1}.
\begin{proof}[Proof of Theorem \ref{thm1}]
Let $\delta>0$ be such that $M_{\delta}\subset \Lambda$, and we show that there exists  $\hat{\e}_{\delta}>0$ such that for any $\e\in (0, \hat{\e}_{\delta})$ and any solution $u\in \widetilde{\mathcal{N}}_{\e}$ of \eqref{MPe}, it holds
\begin{equation}\label{Ua}
\|u\|_{L^{\infty}(\R^{3}\setminus \Lambda_{\e})}<t_{a'}.
\end{equation}
We argue by contradiction, and assume that there is a sequence $\e_{n}\rightarrow 0$, $u_{n}\in \widetilde{\mathcal{N}}_{\e_{n}}$ such that
\begin{equation}\label{AbsAFF}
\|u_{n}\|_{L^{\infty}(\R^{3}\setminus \Lambda_{\e})}\geq t_{a'}.
\end{equation}
Since $J_{\e_{n}}(u_{n})\leq c_{V_{0}}+h_{1}(\e_{n})$, we can argue as in the first part of Lemma \ref{prop3.3} to deduce that $J_{\e_{n}}(u_{n})\rightarrow c_{V_{0}}$.
In view of Lemma \ref{prop3.3}, there exists $(\tilde{y}_{n})\subset \R^{3}$ such that $\e_{n}\tilde{y}_{n}\rightarrow y_{0}$ for some $y_{0} \in M$. 
Take $r>0$ such that, for some subsequence still denoted by itself, it holds $B_{r}(\tilde{y}_{n})\subset \Lambda$ for all $n\in \mathbb{N}$.
Hence $B_{\frac{r}{\e_{n}}}(\tilde{y}_{n})\subset \Lambda_{\e_{n}}$ $n\in \mathbb{N}$. Consequently, 
$$
\R^{3}\setminus \Lambda_{\e_{n}}\subset \R^{3} \setminus B_{\frac{r}{\e_{n}}}(\tilde{y}_{n}) \mbox{ for any } n\in \mathbb{N}.
$$ 
By Lemma \ref{moser}, we can find $R>0$ such that 
$$
v_{n}(x)<t_{a'} \mbox{ for } |x|\geq R, n\in \mathbb{N},
$$ 
where $v_{n}(x)=|u_{\e_{n}}|(x+ \tilde{y}_{n})$ ($v_{n}$ is also strongly convergent in $H^{s}(\R^{3}, \R)$), 
from which we deduce that $|u_{\e_{n}}(x)|<a$ for any $x\in \R^{3}\setminus B_{R}(\tilde{y}_{n})$ and $n\in \mathbb{N}$. Then, there exists $\nu \in \mathbb{N}$ such that for any $n\geq \nu$ and $r/\e_{n}>R$ it holds 
$$
\R^{3}\setminus \Lambda_{\e_{n}}\subset \R^{3} \setminus B_{\frac{r}{\e_{n}}}(\tilde{y}_{n})\subset \R^{3}\setminus B_{R}(\tilde{y}_{n}).
$$ 
Therefore, $|u_{\e_{n}}(x)|<a$ for any $x\in \R^{3}\setminus \Lambda_{\e_{n}}$ and $n\geq \nu$, and this is impossible by \eqref{AbsAFF}.

Let $\tilde{\e}_{\delta}>0$ be given by Theorem \ref{multiple} and we set $\e_{\delta}=\min\{\tilde{\e}_{\delta}, \hat{\e}_{\delta} \}$. Applying Theorem \ref{multiple} we obtain $cat_{M_{\delta}}(M)$ nontrivial solutions to \eqref{MPe}.
If $u\in \h$ is one of these solutions, then $u\in \widetilde{\mathcal{N}}_{\e}$, and in view of \eqref{Ua} and the definition of $g$ we can infer that $u$ is also a solution to \eqref{MPe}. Since $\hat{u}_{\e}(x)=u_{\e}(x/\e)$ is a solution to (\ref{P}), we can infer that \eqref{P} has at least $cat_{M_{\delta}}(M)$ nontrivial solutions.

Finally, we investigate the behavior of the maximum points of  $|\hat{u}_{\e_{n}}|$. Take $\e_{n}\rightarrow 0$ and $(u_{\e_{n}})$ a sequence of solutions to \eqref{MPe} as above. From $(g_1)$, we can find $\gamma>0$ such that
\begin{align}\label{4.18HZ}
g(\e x, t^{2})t^{2}\leq \frac{V_{0}}{2}t^{2}, \mbox{ for all } x\in \R^{3}, |t|\leq \gamma.
\end{align}
Arguing as above, we can find $R>0$ such that
\begin{align}\label{4.19HZ}
\|u_{\e_{n}}\|_{L^{\infty}(B^{c}_{R}(\tilde{y}_{n}))}<\gamma.
\end{align}
Up to a subsequence, we may also assume that
\begin{align}\label{4.20HZ}
\|u_{\e_{n}}\|_{L^{\infty}(B_{R}(\tilde{y}_{n}))}\geq \gamma.
\end{align}
Indeed, if \eqref{4.20HZ} does not hold, we get $\|u_{\e_{n}}\|_{L^{\infty}(\R^{3})}< \gamma$, and using $J_{\e_{n}}'(u_{\e_{n}})=0$, \eqref{4.18HZ} and Lemma \ref{DI} we can deduce that 
\begin{align*}
a[|u_{\e_{n}}|]^{2}+\int_{\R^{3}}V_{0}|u_{\e_{n}}|^{2}dx&\leq \|u_{\e_{n}}\|^{2}_{\e_{n}}+b[u_{\e_{n}}]^{4}_{A_{\e_{n}}}\\
&=\int_{\R^{3}} g_{\e_{n}}(x, |u_{\e_{n}}|^{2})|u_{\e_{n}}|^{2}\,dx\\
&\leq \frac{V_{0}}{2}\int_{\R^{3}}|u_{\e_{n}}|^{2}\, dx.
\end{align*}
This fact yields $\|u_{\e_{n}}\|_{H^{s}(\R^{3})}=0$, which is impossible. Hence, \eqref{4.20HZ} is verified.

In the light of \eqref{4.19HZ} and \eqref{4.20HZ}, we can see that the maximum points $p_{n}$ of $|u_{\e_{n}}|$ belong to $B_{R}(\tilde{y}_{n})$, that is $p_{n}=\tilde{y}_{n}+q_{n}$ for some $q_{n}\in B_{R}$. Since the associated solution of \eqref{P} is of the form $\hat{u}_{n}(x)=u_{\e_{n}}(x/\e_{n})$, we can infer that a maximum point $\eta_{\e_{n}}$ of $|\hat{u}_{n}|$ is $\eta_{\e_{n}}=\e_{n}\tilde{y}_{n}+\e_{n}q_{n}$. Since $q_{n}\in B_{R}$, $\e_{n}\tilde{y}_{n}\rightarrow y_{0}$ and $V(y_{0})=V_{0}$, we can use the continuity of $V$ to deduce that
$$
\lim_{n\rightarrow \infty} V(\eta_{\e_{n}})=V_{0}.
$$
Finally, we provide a decay estimate for $|\hat{u}_{n}|$.
Using Lemma $4.3$ in \cite{FQT}, there exists a function $w$ such that 
\begin{align}\label{HZ1}
0<w(x)\leq \frac{C}{1+|x|^{3+2s}},
\end{align}
and
\begin{align}\label{HZ2}
(-\Delta)^{s} w+\frac{V_{0}}{2(a+bM^{2})}w\geq 0 \mbox{ in } \R^{3}\setminus B_{R_{1}} 
\end{align}
for some suitable $R_{1}>0$, and $M>0$ is such that $a+bM^{2}\geq a+b[u_{n}]^{2}_{A_{\e_{n}}}\geq a+b[v_{n}]^{2}$ (the last inequality is due to Lemma \ref{DI}). By Lemma \ref{moser}, we know that $v_{n}(x)\rightarrow 0$ as $|x|\rightarrow \infty$ uniformly in $n\in \mathbb{N}$, so we can use $(g_1)$ to deduce that there exists $R_{2}>0$ such that
\begin{equation}\label{hzero}
g(\e_{n}x+\e_{n}\tilde{y}_{n}, v_{n}^{2})v_{n}\leq \frac{V_{0}}{2}v_{n}  \mbox{ in } B_{R_{2}}^{c}.
\end{equation}
Arguing as in Lemma \ref{moser}, we can note that $v_{n}$ verifies
\begin{equation}\label{Pkat}
(-\Delta)^{s} v_{n} + \frac{V_{0}}{a+bM^{2}} v_{n}\leq g_{n} \mbox{ in } \R^{3},
\end{equation}
where
$$
g_{n}:=(a+b[v_{n}]^{2})^{-1}[g(\e_{n} x+\e_{n}\tilde{y}_{n}, v_{n}^{2})v_{n}-V(\e_{n} x+\e_{n}\tilde{y}_{n})v_{n}]+\frac{V_{0}}{a+bM^{2}}v_{n}.
$$
Let us denote by $w_{n}$ the unique solution to 
$$
(-\Delta)^{s}w_{n}+\frac{V_{0}}{(a+bM^{2})}w_{n}=g_{n} \mbox{ in } \R^{3}.
$$
%From the arguments used in Lemma \ref{moser}, we can infer that $w_{n}(x)\rightarrow 0$ as $|x|\rightarrow \infty$ uniformly in $n\in \mathbb{N}$ and $0\leq v_{n}\leq w_{n}$ in $\R^{3}$ by comparison. 
By comparison, we have $0\leq v_{n}\leq w_{n}$ in $\R^{3}$ and together with \eqref{hzero} we get
\begin{align*}
(-\Delta)^{s}w_{n}&+\frac{V_{0}}{2(a+bM^{2})}w_{n} \\
&=(-\Delta)^{s}w_{n}+\frac{V_{0}}{(a+bM^{2})}w_{n}-\frac{V_{0}}{2(a+bM^{2})}w_{n} \\
&\leq g_{n}-\frac{V_{0}}{2(a+bM^{2})}v_{n} \\
%&= (a+b[v_{n}]^{2})^{-1}[g(\e_{n} x+\e_{n}\tilde{y}_{n}, v_{n}^{2})v_{n}-V(\e_{n} x+\e_{n}\tilde{y}_{n})v_{n}]+\frac{V_{0}}{a+bM^{2}}v_{n}-\frac{V_{0}}{2(a+bM^{2})}w_{n}\\
&\leq (a+b[v_{n}]^{2})^{-1}[g(\e_{n} x+\e_{n}\tilde{y}_{n}, v_{n}^{2})v_{n}-V(\e_{n} x+\e_{n}\tilde{y}_{n})v_{n}]+\frac{V_{0}}{2(a+bM^{2})}v_{n} \\
&\leq (a+b[v_{n}]^{2})^{-1}\left\{g(\e_{n} x+\e_{n}\tilde{y}_{n}, v_{n}^{2})v_{n}-\left(V(\e_{n} x+\e_{n}\tilde{y}_{n})-\frac{V_{0}}{2}\right)v_{n}\right\} \\
&\leq (a+b[v_{n}]^{2})^{-1}\left\{g(\e_{n} x+\e_{n}\tilde{y}_{n}, v_{n}^{2})v_{n}-\frac{V_{0}}{2} v_{n}\right\} \leq 0 \mbox{ in } B_{R_{2}}^{c}.
\end{align*}
Choose $R_{3}=\max\{R_{1}, R_{2}\}$ and we set 
\begin{align}\label{HZ4}
c=\inf_{B_{R_{3}}} w>0 \mbox{ and } \tilde{w}_{n}=(d+1)w-c w_{n}.
\end{align}
where $d=\sup_{n\in \mathbb{N}} \|w_{n}\|_{L^{\infty}(\R^{3})}<\infty$. 
Our purpose is to verify that
\begin{equation}\label{HZ5}
\tilde{w}_{n}\geq 0 \mbox{ in } \R^{3}.
\end{equation}
Let us note that
\begin{align}
&\lim_{|x|\rightarrow \infty} \sup_{n\in \mathbb{N}}\tilde{w}_{n}(x)=0,  \label{HZ0N} \\
&\tilde{w}_{n}\geq dc+w-dc>0 \mbox{ in } B_{R_{3}} \label{HZ0},\\
&(-\Delta)^{s} \tilde{w}_{n}+\frac{V_{0}}{2(a+bM^{2})}\tilde{w}_{n}\geq 0 \mbox{ in } \R^{3}\setminus B_{R_{3}} \label{HZ00}.
\end{align}
Now, we argue by contradiction. Suppose that there exists a sequence $(\bar{x}_{j, n})\subset \R^{3}$ such that 
\begin{align}\label{HZ6}
\inf_{x\in \R^{3}} \tilde{w}_{n}(x)=\lim_{j\rightarrow \infty} \tilde{w}_{n}(\bar{x}_{j, n})<0. 
\end{align}
From (\ref{HZ0N}), it follows that $(\bar{x}_{j, n})$ is bounded, and, up to subsequence, we may assume that there exists $\bar{x}_{n}\in \R^{3}$ such that $\bar{x}_{j, n}\rightarrow \bar{x}_{n}$ as $j\rightarrow \infty$. 
Thus, (\ref{HZ6}) gives
\begin{align}\label{HZ7}
\inf_{x\in \R^{3}} \tilde{w}_{n}(x)= \tilde{w}_{n}(\bar{x}_{n})<0.
\end{align}
Using the minimality of $\bar{x}_{n}$ and the representation formula for the fractional Laplacian (see Lemma $3.2$ in \cite{DPV}), we can deduce that 
\begin{align}\label{HZ8}
(-\Delta)^{s}\tilde{w}_{n}(\bar{x}_{n})=\frac{c_{3, s}}{2} \int_{\R^{3}} \frac{2\tilde{w}_{n}(\bar{x}_{n})-\tilde{w}_{n}(\bar{x}_{n}+\xi)-\tilde{w}_{n}(\bar{x}_{n}-\xi)}{|\xi|^{3+2s}} d\xi\leq 0.
\end{align}
In view of (\ref{HZ0}) and (\ref{HZ6}), we get $\bar{x}_{n}\in \R^{3}\setminus B_{R_{3}}$,
which together with (\ref{HZ7}) and (\ref{HZ8}) yields
$$
(-\Delta)^{s} \tilde{w}_{n}(\bar{x}_{n})+\frac{V_{0}}{2(a+bM^{2})}\tilde{w}_{n}(\bar{x}_{n})<0,
$$
which is a contradiction due to (\ref{HZ00}).
Hence, (\ref{HZ5}) holds true, and using (\ref{HZ1}) and $v_{n}\leq w_{n}$ we obtain
\begin{align*}
0\leq v_{n}(x)\leq w_{n}(x)\leq \frac{\tilde{C}}{1+|x|^{3+2s}} \mbox{ for all } n\in \mathbb{N}, x\in \R^{3},
\end{align*}
for some constant $\tilde{C}>0$. 
From the definition of $v_{n}$, we have 
\begin{align*}
|\hat{u}_{n}|(x)&=|u_{\e_{n}}|\left(\frac{x}{\e_{n}}\right)=v_{n}\left(\frac{x}{\e_{n}}-\tilde{y}_{n}\right) \\
&\leq \frac{\tilde{C}}{1+|\frac{x}{\e_{n}}-\tilde{y}_{\e_{n}}|^{3+2s}} \\
&=\frac{\tilde{C} \e_{n}^{3+2s}}{\e_{n}^{3+2s}+|x- \e_{n} \tilde{y}_{\e_{n}}|^{3+2s}} \\
&\leq \frac{\tilde{C} \e_{n}^{3+2s}}{\e_{n}^{3+2s}+|x-\eta_{\e_{n}}|^{3+2s}},
\end{align*}
which gives the desired estimate.
\end{proof}

\noindent
%{\bf Acknowledgments.}
%This work was started while the author was visiting the Department of Mathematics of the \'Ecole Polytechnique F\'ed\'erale de Lausanne, whose hospitality he gratefully acknowledges. The author would like to express his gratitude to Prof. Hoai-Minh Nguyen for delightful and pleasant discussions about the results of this work. 

\end{document}